\documentclass[11pt]{amsart}

\usepackage{amsmath,amsthm,amsfonts,latexsym,amscd,amssymb,enumerate,color,hyperref}
\input{xypic}

\usepackage{fullpage}

\usepackage{rotating}
\usepackage{multirow}
\usepackage{graphicx}
\usepackage{hyperref}
\usepackage{epstopdf}
\usepackage[usenames,dvipsnames]{xcolor}
\usepackage{inputenc}
\usepackage{tikz-cd}

\usepackage{booktabs}
\usepackage{enumitem}
\usepackage{multicol}

\usepackage{etoolbox}
\newcounter{dummy}
\usepackage{enumitem} 
\makeatletter  
\newcommand\myitem[1][]{\item[#1]\refstepcounter{dummy}\def\@currentlabel{#1}}   
\makeatother

\usepackage{hyperref}
\hypersetup{colorlinks,allcolors=blue}

\def \<{\langle}
\def \>{\rangle}

\newcommand{\R}{\mathbb{R}}
\newcommand{\NN}{\mathbb{N}}

\newcommand{\GHto}{\stackrel{GH}{\longrightarrow}}
\newcommand{\GHtopair}{\xrightarrow{\protect{GH_{\mathsf{Pair}}}}}
\newcommand{\GHtopointed}{\xrightarrow{GH_*}}

\newcommand{\DD}{\mathcal{D}}
\newcommand{\QQ}{\mathcal{Q}}

\DeclareMathOperator{\asdim}{asdim}

\DeclareMathOperator{\Id}{Id}
\newcommand{\al}{\alpha}

\newcommand{\eps}{\varepsilon}

\newcommand{\vn}{\varnothing}
\newcommand{\monmet}{\mathsf{CMon(Met_*)}}

\newcommand{\mset}[1]{\left\{\!\left\{ #1 \right\}\!\right\}}

\newcommand{\Met}{\mathsf{Met}}
\newcommand{\MetPair}{\mathsf{Met_{Pair}}}
\newcommand{\PMet}{\mathsf{PMet}}

\DeclareMathOperator{\dist}{dist}
\DeclareMathOperator{\diam}{diam}

\DeclareMathOperator{\curv}{curv}

\newtheorem{thm}{Theorem}[section]
\newtheorem{cor}[thm]{Corollary}
\newtheorem{lem}[thm]{Lemma}
\newtheorem{prop}[thm]{Proposition}

\newtheorem{theorem}{Theorem}

\newtheorem{corollary}[theorem]{Corollary}

\theoremstyle{definition}
\newtheorem{defn}[thm]{Definition}

\newtheorem{example}[thm]{Example}
\newtheorem{rem}[thm]{Remark}

\newtheorem*{ack}{Acknowledgements}

\numberwithin{equation}{section}
\numberwithin{table}{section}

\makeatletter
\def\@setthanks{\vspace{-\baselineskip}\def\thanks##1{\@par##1}\thankses}
\makeatother

\begin{document}

\author[M.~Che]{Mauricio Che $^\mathrm{A}$}
\address[Che]{Department of Mathematical Sciences, Durham University, United Kingdom.}
\email{mauricio.a.che-moguel@durham.ac.uk}

\thanks{$^{\mathrm{A}}$Supported by CONACYT Doctoral Scholarship No. 769708.}

\author[F.~Galaz-Garc\'ia]{Fernando Galaz-Garc\'ia $^{\mathrm{B}}$}
\address[Galaz-Garc\'ia]{Department of Mathematical Sciences, Durham University, United Kingdom.}
\email{fernando.galaz-garcia@durham.ac.uk}

\author[L.~Guijarro]{Luis Guijarro $^{\mathrm{B},\mathrm{C}}$}
\address[Guijarro]{Department of Mathematics, Universidad Aut\'onoma de Madrid and ICMAT CSIC-UAM-UC3M, Spain}
\email{luis.guijarro@uam.es} 

\thanks{$^{\mathrm{B}}$Supported in part by research grants MTM2017-‐85934-‐C3-‐2-‐P from the  Ministerio de Econom\'{\i}a y Competitividad de Espa\~{na} (MINECO) and PID2021-124195NB-C32 from the Ministerio de Ciencia e Innovación (MICINN).}

\thanks{$^{\mathrm{C}}$Supported in part by QUAMAP - Quasiconformal Methods in Analysis and Applications (ERC grant 834728), and by ICMAT Severo Ochoa project CEX2019-000904-S (MINECO).}

\author[I.~Membrillo Solis]{Ingrid Amaranta Membrillo Solis $^\mathrm{D}$}
\address[Membrillo Solis]{Mathematical Sciences, University of Southampton, United Kingdom}
\email{i.membrillo-solis@soton.ac.uk} 
\curraddr{School of Mathematical Sciences, Queen Mary University of London, United Kingdom}
\email{i.a.membrillosolis@qmul.ac.uk}

\thanks{$^\mathrm{D}$Supported by the Leverhulme Trust (grant RPG-2019-055).}

\title[Metric Geometry of Persistence Diagrams]{Metric Geometry of Spaces of Persistence Diagrams}
\date{\today}

\begin{abstract}
Persistence diagrams are objects that play a central role in topological data analysis. In the present article, we investigate the local and global geometric properties of spaces of persistence diagrams. In order to do this, we construct a family of functors $\DD_p$, $1\leq p \leq\infty$,  that  assign, to each metric pair $(X,A)$, a pointed metric space $\DD_p(X,A)$. Moreover, we show that $\DD_{\infty}$ is sequentially continuous with respect to the Gromov--Hausdorff convergence of metric pairs, and we prove that $\DD_p$ preserves several useful metric properties, such as completeness and separability, for $p \in [1,\infty)$, and geodesicity and non-negative curvature in the sense of Alexandrov, for $p=2$. For the latter case, we describe the metric of the space of directions at the empty diagram. We also show that the Fr\'echet mean set of a Borel probability measure on $\DD_p(X,A)$, $1\leq p \leq\infty$, with finite second moment and compact support is non-empty. As an application of our geometric framework, we prove that the space of Euclidean persistence diagrams, $\DD_{p}(\R^{2n},\Delta_n)$,  $1\leq n$ and $1\leq p<\infty$, has infinite covering, Hausdorff, asymptotic, Assouad, and Assouad--Nagata dimensions.
\end{abstract}

\subjclass[2020]{53C23, 55N31, 54F45,64R40}

\keywords{Alexandrov spaces, asymptotic dimension, metric pairs, Gromov--Hausdorff convergence, persistence diagram, Fréchet mean set}
\maketitle

\setcounter{tocdepth}{1}
\tableofcontents

\section{Introduction}
\label{S:INTRO}

After first appearing in the pioneering work of Edelsbrunner, Letscher, and Zomorodian \cite{ELZ00},  in recent years, persistent homology has become an important tool in the analysis of scientific datasets, covering a wide range of applications \cite{ARC14,buchet18,E08,kov16,munch13,PDeM15,zhu13} and playing a central role in topological data analysis. 

In \cite{ZC05}, Carlsson and Zomorodian introduced persistence modules indexed by the natural numbers as the algebraic objects underlying persistent homology. 
The successful application of persistent homology in data analysis is, to a great extent, due to the notion of persistence diagrams. These were introduced by Cohen-Steiner, Edelsbrunner, and Harer as equivalent representations for persistence modules indexed by the positive real numbers \cite{C-SEH07}. More precisely, a \textit{persistence diagram} is a multiset of points $(b,d)\in\overline{\mathbb{R}}^2_{\geq 0}$, where  $\overline {\mathbb R}^2_{\geq 0}=\{(x,y)\in\mathbb \overline \overline{\mathbb R}\times \overline{\mathbb R}: 0\leq x< y\}$ and $\overline{\mathbb R}=\mathbb R\cup \{-\infty,\infty\}$.  Persistence diagrams are objects that, in a certain sense, are easier to visualize than persistence modules. Moreover, the set of persistence diagrams supports a family of metrics, called \emph{$p$-Wasserstein metrics}, parametrized by $1\leq p\leq \infty$ (see \cite{CEHM10}), with the metric corresponding to $p=\infty$ also known as the \emph{bottleneck distance}. We will denote by $\DD_p(\mathbb R^2,\Delta)$, where $\Delta$ is the diagonal of $\R^2$, the metric space defined by the set of persistence diagrams that arise in persistent homology equipped with the $p$-Wasserstein metric.

Several authors have extensively studied the geometry and topology of the spaces $\DD_p(\mathbb R^2,\Delta)$. For instance, Mileyko, Mukherjee, and Harer \cite{mileyko1} examined the completeness, separability, and compactness of subsets of the space $\DD_p(\mathbb R^2,\Delta)$ with the $p$-Wasserstein metric, $1\leq p<\infty$. Turner et al.\  \cite{turner14} showed that $\DD_2(\mathbb R^2,\Delta)$ 
is an (infinite dimensional) Alexandrov space with non-negative curvature. The results in \cite{mileyko1} imply the existence of Fr\'echet means for certain probability distributions on $\DD_p(\mathbb R^2,\Delta)$. For $p=2$, the results in \cite{turner14} imply the convergence of certain algorithms used to find Fr\'echet means of finite sets in $\DD_p(\mathbb R^2,\Delta)$. Turner \cite{turner20} studied further statistical properties, such as the median of finite sets in $\DD_p(\mathbb R^2,\Delta)$, and its relation to the mean.  All these results are crucially based on the presence of the $p$-Wasserstein metric and, when $p=2$, on the Alexandrov space structure.

\subsection*{Our contributions}  Motivated by the preceding considerations, we develop a general framework for the geometric study of generalized spaces of persistence diagrams.
To the best of our knowledge, the present article is the first attempt to systematically analyze the geometric properties of such spaces. Our departure point is the existence of a family of functors $\DD_p\colon \MetPair\to \Met_*$, $1\leq p < \infty$ (resp.\ $\DD_\infty\colon \MetPair\to \PMet_*$), from the category $\MetPair$ of metric pairs equipped with relative Lipschitz maps into the category $\Met_*$ of pointed metric spaces equipped with pointed Lipschitz maps (resp.\ the category $\PMet_*$ of pointed pseudometric spaces equipped with pointed Lipschitz maps), which assign to each metric pair $(X,A)$, where $X$ is a metric space and $A\subset X$ is a closed and non-empty subset, a space of persistence diagrams $\DD_p(X,A)$. In particular,  for $(X,A) = (\R^2,\Delta)$,  where  $\Delta = \{(x,y)\in \R^2: x=y\}$ is the diagonal of $\R^2$, we recover the spaces $\DD_p(\R^2,\Delta)$,  that arise in persistent homology, equipped with the $p$-Wasserstein distance. These spaces were studied in \cite{mileyko1,turner14}.  
Bubenik and Elchesen studied such functors from an algebraic point of view in \cite{bubenik2}. Here, we disregard the algebraic structure  and focus on the behavior of several basic topological, metric, and geometric properties and invariants under the functors $\DD_p$. When $X=\R^{2n}_{\geq 0} = \{(x_1,y_1,\dots,x_n,y_n):0\leq x_i\leq y_i\ \text{for}\ i=1,\dots,n\}$ and $A=\Delta_n = \{(x_1,y_1,\dots,x_n,y_n)\in \R^{2n}_{\geq 0}: x_i=y_i\ \text{for}\ i=1,\dots, n\}$ for $n\geq 2$, the resulting spaces $\DD_p(\R^{2n}_{\geq 0},\Delta_n)$, that we call from now on \emph{spaces of Euclidean persistence diagrams}, can be considered as the parameter spaces for rectangle persistent modules. These modules arise in the context of multiparameter persistent homology and have been investigated by several authors (see, for example, \cite{bak2021,botnan2022,cochoy2020,SC17}).

Our aim is twofold: first, to show that many basic results useful in statistical analysis on spaces of Euclidean persistence diagrams hold for the generalized persistence diagram spaces  $\DD_p(X,A)$, 
and, second, to study the intrinsic geometry of such spaces, which is of interest in its own right. As an application of our framework, we show that different notions of dimension are infinite for spaces the spaces of Euclidean diagrams $\DD_p(\R^{2n},\Delta_n)$.

Given that each metric pair $(X,A)$ gives rise to a pointed metric space $\DD_p(X,A)$,  it is natural to ask whether some form of continuity holds with respect to $(X,A)$. To address this question, we introduce the \emph{Gromov--Hausdorff convergence of metric pairs}, a mild generalization of the usual Gromov--Hausdorff convergence of pointed metric spaces (see Definition~\ref{d:pairsGHconv}), and obtain our first main result.

\begin{theorem}
\label{t:continuity}
The functor $\DD_p$
, $1\leq p\leq \infty$, is sequentially continuous with respect to the Gromov--Hausdorff convergence of metric pairs if and only if  $p=\infty$.
\end{theorem}

One may think of Theorem~\ref{t:continuity} as providing formal justification for using persistence diagrams calculated by computers in applications. Since computers have finite precision,
 such diagrams are elements of a discrete space that approximates the ideal space of persistence diagrams in the Gromov--Hausdorff sense.
Theorem~A ensures that this approximation is continuous. In particular, a small perturbation of the space of parameters $(X,A)$ (which is $(\R^2,\Delta)$ in the Euclidean case)  will result in a small  perturbation in the corresponding space of persistence diagrams.

Our second  main result is the invariance of several basic metric properties under $\DD_p$, $1\leq p<\infty$,  generalizing results in \cite{mileyko1,turner14} for spaces of Euclidean persistence diagrams. These properties include completeness and geodesicity (which we require of Alexandrov spaces), as well as non-negative curvature when  $p=2$. 
    
\begin{theorem}
\label{t:properties_Dp}
Let $(X,A)\in \MetPair$ and let $1\leq p< \infty$. Then the following assertions hold:
\begin{enumerate}
    \item If $X$ is complete, then $\DD_p(X,A)$ is complete.
    \label{it:completeness}
    \item If $X$ is separable, then $\DD_p(X,A)$ is  separable.
    \label{it:separability}
    \item If $X$ is a proper geodesic space, then $\DD_p(X,A)$ is a geodesic space.
    \label{it:geodesicity}
    \item If $X$ is a proper Alexandrov space with non-negative curvature, then $\DD_2(X,A)$ is an Alexandrov space with non-negative curvature.
    \label{it:nncurvature}
\end{enumerate}
\end{theorem}
One may show that $\DD_p(X,A)$ is complete if and only if $X/A$ is complete for any $p\in [1,\infty]$ (see \cite{CGGGMSV2022}). The behavior of $\DD_\infty$ is substantially different to that of $\DD_p$, $1\leq p<\infty$. Indeed, Theorem~\ref{t:properties_Dp} 
fails for $\DD_\infty$, as shown in \cite{CGGGMSV2022}.

The functor $\DD_p$ allows us to carry over, with some minor modifications, the Euclidean proofs  of completeness and separability in \cite{mileyko1} to prove items \eqref{it:completeness} and \eqref{it:separability} in Theorem~\ref{t:properties_Dp}. Items~\eqref{it:geodesicity} and~\eqref{it:nncurvature} generalize the corresponding Euclidean results in \cite{turner14}, asserting that the spaces $\DD_p(\R^2,\Delta)$, $1\leq p<\infty$, are geodesic and that $\DD_2(\R^2,\Delta)$ is an Alexandrov space of non-negative curvature.   Our proofs differ from those in \cite{turner14} and rely on a characterization of geodesics in persistence diagram spaces originally  obtained by Chowdhury in \cite{chowdhury} in the Euclidean case.

Motivated by results in \cite{turner14} on the Alexandrov space $\DD_2(\R^2,\Delta)$, we analyze  the infinitesimal geometry of Alexandrov spaces arising via Theorem~\ref{t:properties_Dp}~\eqref{it:nncurvature} at their distinguished point, the empty diagram $\sigma_{\vn}$. The infinitesimal structure of an Alexandrov space at a point is captured by the \emph{space of directions}, which is itself a metric space and corresponds to the unit tangent sphere in the case of Riemannian manifolds. 

First, we show that the space of directions $\Sigma_{\sigma_\vn}$ at $\sigma_\vn \in \DD_2(X,A)$  has diameter at most $\pi/2$ (Proposition~\ref{prop:diameter.space.of.directions}). Second, we show that directions in $\Sigma_{\sigma_\vn}$ corresponding to finite diagrams are dense in $\Sigma_{\sigma_\vn}$ (Proposition~\ref{prop:directions.to.finite.diagrams}). Finally, we use this to obtain an explicit description of the metric structure of $\Sigma_{\sigma_{\vn}}$. These results are new, even in the case of Euclidean persistence diagrams.

\begin{theorem}
\label{t:space.of.directions.empty.diagram}
\label{prop:explicit.space.of.directions} The space of persistence diagrams $\Sigma_{\sigma_\vn}$ at $\sigma_\vn \in \DD_2(X,A)$  has diameter at most $\pi/2$ and directions in $\Sigma_{\sigma_\vn}$ corresponding to finite diagrams are dense in $\Sigma_{\sigma_\vn}$. Moreover, 
consider elements in $\Sigma_{\sigma_\varnothing}$ given by geodesics $\xi_\sigma= \{\xi_a\}_{a\in\sigma}$ and $\xi_{\sigma'}= \{\xi_{a'}\}_{a'\in\sigma'}$ joining $\sigma_\varnothing$ with $\sigma,\sigma'\in \DD_2(X,A)$, where $\xi_a$ and $\xi_{a'}$ are geodesics in $X$ joining points in $A$ with $a$ and $a'$, respectively. Then
\[
d_2(\sigma,\sigma_\varnothing)d_2(\sigma',\sigma_\varnothing)\cos\angle(\xi_{\sigma},\xi_{\sigma'})  = \sup_{\phi\colon\tau\to\tau'} \sum_{a \in \tau} d(a,A)d(a',A)\cos\angle(\xi_a,\xi_{\phi(a)}),
\]
where $d$ is the metric on $X$, $d_2$ is the $2$-Wasserstein metric, and $\phi$ ranges over all bijections between subsets $\tau$ and $\tau'$ of 
points in $\sigma$ and $\sigma'$, respectively, such that $\xi_a(0) = \xi_{\phi(a)}(0)$ for all $a \in \tau$.
\end{theorem}

In $\DD_2(\R^2,\Delta)$, persistence diagrams in a neighborhood of the empty diagram may be thought of as coming from noise. Thus, the space of directions at the empty diagram may be interpreted as directions coming from noise. Theorem~\ref{prop:explicit.space.of.directions} implies that the geometry at the empty diagram is singular, as directions at this point make an angle of at most $\pi/2$. In particular, the infinitesimal geometry is not Euclidean. Hence, embedding noisy sets of persistence diagrams into a Hilbert space might require large metric distortions. This might be of relevance in the vectorization of sets of persistence diagrams where noise might be present, which in turn plays a role in applying machine learning to such sets (see, for example, \cite{bubenik.persistence.landscapes,carriere_et_al:LIPIcs.SoCG.2019.21,kusano.et.al.machine.learning}).

A further advantage of the metric pair framework is that it yields the existence of Fr\'echet mean sets for certain classes of probability measures on generalized spaces of persistence diagrams, as shown in \cite{mileyko1} in the Euclidean case. The elements of Fr\'echet mean sets (which, a priori, may be empty) are also called \emph{barycenters} (see, for example, \cite{Afsari}) and may be interpreted as centers of mass of the given probability measure. In the case of the spaces $\DD_p(X,A)$, we may interpret a finite collection of persistence diagrams as a measure $\mu$ on $\DD_p(X,A)$ with finite support. An element of the corresponding Fr\'echet mean set then may be interpreted as an average of the diagrams determining the measure. For the spaces $\DD_p(X,A)$, we have the following result.

\begin{theorem}\label{t:frechet means}
Let $\mu$ be a Borel probability measure on $\DD_p(X,A)$, $1\leq p<\infty$. Then the following assertions hold:
\begin{enumerate}
\item If $\mu$ has finite second moment and compact support, then the Fr\'echet mean set of $\mu$ is non-empty.
\item If $\mu$ is tight and has rate of decay at infinity $q>\max\{2,p\}$, then the Fr\'echet mean set of $\mu$ is non-empty.
\end{enumerate}
\end{theorem}

The proof of this theorem follows along the lines of the proofs of the corresponding Euclidean statements in \cite{mileyko1}, and hinges on the fact, easily shown, that the characterization of totally bounded subsets of spaces of Euclidean persistence diagrams given in \cite{mileyko1} also holds in the setting of metric pairs (see Proposition \ref{prop:totally bounded sets}).

As mentioned above, our constructions include, as a special case, spaces of Euclidean persistence diagrams, such as the classical space  $\DD_p(\R^{2},\Delta)$. Our fourth main result shows that different notions of dimension for this space are all infinite. We let $\Delta_n = \{ (v,v) \in \R^{2n}: v \in \mathbb{R}^n \}$.

\begin{theorem}
\label{t:dimensions}
The space $\DD_p(\R^{2n},\Delta_n)$, $1\leq n$ and $1\leq p<\infty$, has infinite covering, Hausdorff, asymptotic, Assouad, and Assouad--Nagata dimension.
\end{theorem}

It is known that every metric space of finite asymptotic dimension admits a coarse embedding into some Hilbert space (see \cite[Example 11.5]{Roe}). The spaces $\DD_{p}(\R^{2},\Delta)$, $2<p\leq \infty$, do not admit such  embeddings (see \cite[Theorem 21]{BW20} and \cite[Theorem 3.2]{Wagner21}). Hence, their asymptotic dimension is infinite (cf.\ \cite[Corollary 27]{BW20}). These observations, along with Theorem~\ref{t:dimensions}, immediately imply the following.

\begin{corollary}
\label{COR:infinite.asymptotic.dimension}
The space $\DD_{p}(\R^{2},\Delta)$, $1\leq p\leq \infty$, has infinite asymptotic dimension.
\end{corollary}

Our analysis also shows that other spaces of Euclidean persistence diagrams appearing in topological data analysis, such as 
$\DD_{p}(\mathbb R^{2n}_+,\Delta_n)$ and $\DD_{p}(\mathbb R^{2n}_{\geq0},\Delta_n)$, have infinite asymptotic dimension as well (see Section~\ref{s:Euclidean diagrams} for precise definitions). One may think of these spaces  as parameter spaces for persistence rectangles in multidimensional persistent homology \cite{bak2021,SC17}. We point out that the proof of Theorem~\ref{t:dimensions} is based on different arguments to those in \cite{BW20,Wagner21} and provides a unified approach for all $1\leq p<\infty$. The crucial point in our proof is the general observation that, if a metric pair $(X,A)$ contains a curve whose distance to $A$ grows linearly, then $\DD_{p}(X,A)$ has infinite asymptotic dimension (see Proposition~\ref{cor:asdim-rayemb}). 
We point out that our arguments to prove Theorem~\ref{t:dimensions} can be used to prove an analogous result for the spaces of persistence diagrams with finitely (but arbitrarily) many points equipped with the $p$-Wasserstein distance, $1\leq p<\infty$, as defined, for example, in \cite{bubenik2}. Thus, all such spaces have infinite Hausdorff, covering, asymptotic, Assouad, and Assouad--Nagata dimensions (see Corollary~\ref{c:dimensions.finite.spaces}).

Note that Theorem~\ref{t:properties_Dp}~\eqref{it:nncurvature} provides a systematic way of constructing examples of  Alexandrov spaces of non-negative curvature. In particular, by Theorem~\ref{t:dimensions}, the space $\DD_2(\R^{2n},\Delta_n)$ is an infinite-dimensional Alexandrov space. 
In contrast to the finite-dimensional case, where every Alexandrov space is proper, there are few results about infinite-dimensional Alexandrov spaces in the literature (see, for example, \cite[Section 13]{Plaut} and, more recently, \cite{mitsuishi,Yo2012,Yo2014}). Technical difficulties occur that do not arise in finite dimensions (see, for example, \cite{Ha}), and a more thorough understanding of the infinite-dimensional case is lacking.

\subsection*{Related work}
Different generalizations of persistence diagrams have appeared in the persistent homology literature, as well as their corresponding spaces of persistence diagrams. In \cite{Pat18}, Patel generalizes persistent diagram invariants of persistence modules to cases where the invariants are associated to functors from a poset $P$ to a symmetric monoidal category. In \cite{kim21}, Kim and Mémoli define the notion of \emph{rank invariant} for functors with indices in an arbitrary poset, which allows defining persistence diagrams for any persistence module $F$ over a poset regardless of whether $F$ is interval-decomposable or not. In \cite{divol21}, Divol and Lacombe considered a persistence diagram as a discrete measure, expressing the distance between persistence diagrams as an optimal transport problem. In this context, the authors introduced Radon measures supported on the upper half plane, generalizing the notion of persistence diagrams, and studied the geometric and topological properties of spaces of Radon measures.
Bubenik and Elchesen considered a functor in \cite{bubenik2}, which sends metric pairs to free commutative pointed metric monoids and studied many algebraic properties of such a functor.

In \cite{bubenik-hartsock}, which followed the first version of the present article, Bubenik and Hartsock studied topological and geometric properties of spaces of persistence diagrams and also considered the setting of pairs $(X, A)$. To address the existence of optimal matchings and geodesics, non-negative curvature in the sense of Alexandrov, and the Hausdorff and asymptotic dimension of spaces of persistence diagrams, Bubenik and Hartsock require the set $A\subset X$ to be \emph{distance minimizing}, i.e., for all $x\in X$, there exists $a \in A$ such
that $\dist(x, A) = d(x, a)$. This property holds when $X$ is proper and $A$ is closed, which we assume in items  \eqref{it:geodesicity} and \eqref{it:nncurvature} of Theorem~\ref{t:properties_Dp}. The authors of \cite{bubenik-hartsock} also show that $\DD_p(X, A)$ has infinite asymptotic dimension when $X$ is geodesic and proper, $X/A$ is unbounded, and $A$ is distance minimizing. The spaces of Euclidean persistence diagrams on $n\geq 1$ points equipped with the $p$-Wasserstein distance, $1\leq p\leq \infty$, have finite asymptotic dimension and therefore admit a coarse embedding into a Hilbert space (see \cite{MitraVirk}). On the other hand, the space of Euclidean persistence diagrams on finitely many points equipped with the bottleneck distance has infinite asymptotic dimension \cite[Corollary 27]{bubenik-hartsock} and  cannot be coarsely embedded into a Hilbert space (see \cite[Theorem 4.3]{MitraVirk}). Bubenik and Hartsock have extended these results to metric pairs in \cite{bubenik-hartsock}. Carri\`ere and Bauer have studied the Assouad dimension and bi-Lipschitz embeddings of spaces of finite persistence diagrams in \cite{carriere_et_al:LIPIcs.SoCG.2019.21}. More recently, Bate and Garcia-Pulido have shown in \cite{bate2024bilipschitz} that the space of persistence barcodes with at most
$m$-points can be bi-Lipschitz embedded into $\ell_2$. They point out that their results also hold for generalized persistence diagrams as considered in the present article whenever $X = \overline{\Omega}$, $A=\partial \Omega$, and $\Omega$ is a proper, open subset of $\mathbb{R}^n$.

With respect to Fr\'echet means of probability measures defined on the spaces of persistence diagrams, Divol and Lacombe in \cite{divol21} investigated the existence of such Fr\'echet means for probability measures defined on the space of persistence measures in $(\R^2,\Delta)$ equipped with the optimal partial transport, which in particular contain the spaces $\DD_p(\R^2,\Delta)$. Our results, although more particular in the hypotheses that we impose on the probability measures considered, are more general with respect to the spaces they are defined on.

\subsection*{Organization.} Our article is organized as follows. In Section~\ref{S:PRELIMINARIES}, we present the background on metric pairs, metric monoids, and Alexandrov spaces, and introduce the functor $\DD_p$. In Section~\ref{S:CONTINUITY}, we define Gromov--Hausdorff convergence for metric pairs and prove Theorem~\ref{t:continuity}. The proofs of items \eqref{it:completeness} and \eqref{it:separability} in Theorem~\ref{t:properties_Dp} and of Theorem~\ref{t:frechet means} follow, with minor modifications, along the same lines as those for the corresponding statements in the Euclidean case. For the sake of completeness, we have included a full treatment of these results in Appendix~\ref{app:completeness.separability.frechet.means}. 
In Section \ref{s:geodesicity}, we analyze the geodesicity of the spaces $\DD_p(X,A)$ and prove item \eqref{it:geodesicity} of Theorem \ref{t:properties_Dp}. In Section \ref{s:curvature_bounds}, we analyze the existence of lower curvature bounds for our spaces of persistence diagrams and prove item~\eqref{it:nncurvature} of Theorem~\ref{t:properties_Dp}. In Section~\ref{S:SPACES_OF_DIRECTIONS}, we make some remarks about their local structure. 
Finally, in Section~\ref{s:Euclidean diagrams} we specialize our constructions to the spaces of Euclidean persistence diagrams, which include the classical  space of persistence diagrams, and prove Theorem~\ref{t:dimensions} 
(cf.\ Corollary~\ref{C:INFINITE_DIMENSIONS}).

\begin{ack}
Luis Guijarro and Ingrid Amaranta Membrillo Solis would like to thank the Department of Mathematical Sciences at Durham University for its hospitality while part of this paper was written. Fernando Galaz-Garc\'ia would like to thank Marvin Karsunky for his help at the initial stages of this project. The authors would like to thank Samir Chowdhury, Tristan Madeleine, and Motiejus Valiunas for their helpful comments on a previous version of this article. Finally, the authors would also like to thank one of the referees for their detailed comments on the manuscript and suggestions on how to improve some of our results in Section~\ref{s:Euclidean diagrams}, leding to the strengthened statement of Theorem~\ref{t:dimensions}.
\end{ack}


\section{Preliminaries}
\label{S:PRELIMINARIES}

In this section, we collect preliminary material that we will use in the rest of the article and prove some elementary results on the spaces of persistence diagrams. 
Our primary reference for metric geometry will be \cite{BBI}. 


\subsection{Metric pairs}
Let $X$ be a set. A map $d\colon X\times X\to[0,\infty)$ is a \emph{metric} on $X$ if $d$ is symmetric, satisfies the triangle inequality, and is  definite, i.e. $d(x,y) = 0$ if and only if $x=y$. 
A  \textit{pseudometric space} is defined similarly;  while keeping the other properties, and still requiring that $d(x,x) = 0$ for all $x\in X$, we allow for points $x,y$ in $X$ with $x\neq y$ and $d(x,y)=0$, in which case $d$ is a \textit{pseudometric}. We  obtain \textit{extended metric} and \textit{extended pseudometric spaces} if we allow for $d$ to take the value $\infty$. Note that when $d$ is a pseudometric, points at distance zero from each other give a partition of $X$, and $d$ induces a metric in the corresponding  quotient set.

Let $(X,d_X)$, $(Y,d_Y)$ be two extended pseudometric spaces. A \emph{Lipschitz map} $f\colon X\to Y$ with \emph{Lipschitz constant} $C$ is a map such that $d_Y(f(x),f(x'))\leq C\cdot d_X(x,x')$ for all $x,x'\in X$ and $y,y'\in Y$. 


\begin{defn}\label{d:metric pairs}
Let $\mathsf{Met_{Pair}}$ denote the category of metric pairs, whose objects,  $\text{Obj}(\mathsf{Met_{Pair}})$, are pairs $(X,A)$ such that $(X,d_X)$ is a metric space and $A\subseteq X$ is closed and non-empty, and whose morphisms, $\text{Hom}(\mathsf{Met_{Pair}})$, are \emph{relative Lipschitz maps}, i.e.\ Lipschitz maps $f\colon (X,A)\to(Y,B)$ such that $f(A)\subseteq B$.  When $A$ is a point, we will talk about \emph{pointed metric spaces} and \emph{pointed Lipschitz maps}, i.e.\ Lipschitz maps $f\colon (X,\{x\})\to(Y,\{y\})$ such that $f(x)=y$. We will denote the category of pointed metric spaces by $\mathsf{Met}_*$.
Similarly, we define the category $\mathsf{PMet_{*}}$ of pointed pseudometric spaces, whose objects $\text{Obj}(\mathsf{PMet}_{*})$, are pairs $(D,\{\sigma\})$ such that $D$ is a pseudometric space and $\sigma$ is a point in $D$. The morphisms of $\mathsf{PMet}_{*}$ are pointed Lipschitz maps. 
\end{defn}


\subsection{Commutative metric monoids and spaces of persistence diagrams}

Some of the definitions and results in this subsection may be found in \cite{bubenik2,bubenik1}. For completeness, we provide full proofs of all the statements. We will denote multisets by using two curly brackets $\mset{\cdot}$ and will usually denote persistence diagrams by Greek letters.

Let $(X,d)$ be a metric space and fix $p\in[1,\infty]$. We define the space $(\widetilde{\DD}(X), \widetilde d_p)$ on $X$ as the set of countable multisets $\mset{x_1,x_2,\dots}$ of elements of $X$ equipped with the  \emph{$p$-Wasserstein pseudometric} $\widetilde d_p$, which is given by 
\begin{equation}
\widetilde d_p(\widetilde\sigma,\widetilde\tau)^p=\inf_{\phi\colon \widetilde\sigma \to \widetilde\tau} \sum_{x \in \widetilde\sigma} d(x,\phi(x))^p
\end{equation}
if $p < \infty$, and
\begin{equation}
\widetilde d_p(\widetilde\sigma,\widetilde\tau)=\inf_{\phi\colon \widetilde\sigma \to \widetilde\tau} \sup_{x \in \widetilde\sigma} d(x,\phi(x))
\end{equation}
if $p = \infty$,
where $\phi$ ranges over all bijections between $\widetilde\sigma$ and $\widetilde\tau$ in $\widetilde{\DD}(X)$. Here, by convention, we set $\inf \varnothing = \infty$, that is, we have $\widetilde d_p(\widetilde\sigma,\widetilde\tau) = \infty$ whenever $\widetilde\sigma$ and $\widetilde\tau$ do not have the same cardinality. 

The function $\widetilde{d}_p$ defines an extended pseudometric in $\widetilde{\DD}(X)$, since it is clearly non-negative, symmetric, and the triangle inequality may be proved as follows: if $\widetilde{\rho},\widetilde{\sigma},\widetilde{\tau}\in\widetilde{\DD}(X)$ have the same cardinality and $\phi\colon \widetilde{\rho}\to \widetilde{\sigma}$ and $\psi\colon \widetilde{\sigma}\to \widetilde{\tau}$ are bijections, then $\psi\circ \phi\colon \widetilde{\rho}\to \widetilde{\tau}$ is also a bijection and, if $p<\infty$, then
\begin{align*}
    \widetilde d_p(\widetilde \rho,\widetilde \tau) &\leq \left(\sum_{x\in \widetilde\rho} d(x,\psi\circ\phi(x))^p\right)^{1/p}\\
    &\leq \left(\sum_{x\in \widetilde\rho} (d(x,\phi(x))+d(\phi(x),\psi\circ\phi(x)))^p\right)^{1/p}\\
    &\leq \left(\sum_{x\in \widetilde\rho} d(x,\phi(x))^p\right)^{1/p}+\left(\sum_{x\in \widetilde\rho} d(\phi(x),\psi\circ\phi(x))^p\right)^{1/p}\\
    &= \left(\sum_{x\in \widetilde\rho} d(x,\phi(x))^p\right)^{1/p}+\left(\sum_{y\in \widetilde\sigma} d(y,\psi(y))^p\right)^{1/p}.
\end{align*}
Taking the infimum over bijections $\phi$ and $\psi$ we get the claim. If the cardinalities of $\widetilde \rho, \widetilde\sigma, \widetilde\tau$ are not the same, the inequality is trivial, since both sides or just the right-hand side would be infinite. For $p=\infty$ the argument is analogous and easier.

Given two multisets $\widetilde\sigma$ and $\widetilde\tau$, we define their \emph{sum} $\widetilde\sigma+\widetilde\tau$ to be their disjoint union. We can make $\widetilde{\DD}(X)$ into a commutative monoid with monoid operation given by taking sums of multisets, and with identity $\widetilde{\sigma}_{\varnothing}$ the empty multiset. It is easy to check that $\widetilde{d}_p$ is \emph{(left-)contractive}, that is, $\widetilde{d}_p(\widetilde\sigma,\widetilde\tau) \geq \widetilde{d}_p(\widetilde\rho+\widetilde\sigma,\widetilde\rho+\widetilde\tau)$ for all $\widetilde\sigma,\widetilde\tau,\widetilde\rho \in \widetilde{\DD}(X)$.

From now on, let $(X,A)\in \mathsf{Met_{Pair}}$. Given $\widetilde\sigma,\widetilde\tau \in \widetilde{\DD}(X)$, we write $\widetilde\sigma \sim_A \widetilde\tau$ if there exist $\widetilde\alpha,\widetilde\beta \in \widetilde{\DD}(A)$ such that $\widetilde\sigma+\widetilde\alpha = \widetilde\tau+\widetilde\beta$. It is easy to verify that $\sim_A$ defines an equivalence relation on $\widetilde{\DD}(X)$ such that, if $\widetilde{\alpha}_1\sim_A \widetilde{\alpha}_2$ and $\widetilde{\beta}_1\sim_A \widetilde{\beta}_2$, then $\widetilde{\alpha}_1+\widetilde{\beta}_1 \sim_A\widetilde{\alpha}_2+\widetilde{\beta}_2$, i.e. $\sim_A$ is a \textit{congruence relation} on $\widetilde{\DD}(X)$ (see, for example, \cite[p.\ 27]{hungerford}).
We denote by $\DD(X,A)$ the quotient set $\widetilde{\DD}(X)/{\sim_A}$. Given $\widetilde\sigma \in \widetilde{\DD}(X)$, we write $\sigma$ for the equivalence class of $\widetilde\sigma$ in $\DD(X,A)$. Note that $\widetilde \sigma \sim_A \widetilde \tau$ if and only if $\widetilde \sigma \setminus A = \widetilde \tau \setminus A$, that is, $\widetilde\sigma$ and $\widetilde\tau$ share the same points with the same multiplicities outside $A$. The monoid operation on $\widetilde{\DD}(X)$ induces a monoid operation on $\DD(X,A)$ by defining $ \sigma + \tau $ as the congruence class corresponding to $\widetilde\sigma + \widetilde\tau$.

The function $\widetilde{d}_p$ on $\widetilde{\DD}(X)$ induces a non-negative function $d_p\colon \DD(X,A) \times \DD(X,A) \to [0,\infty]$ defined by
\begin{equation}
d_p(\sigma,\tau)= \inf_{\widetilde\alpha,\widetilde\beta\in\widetilde{\DD}(A)}\widetilde d_p(\widetilde\sigma+\widetilde\alpha,\widetilde\tau+\widetilde\beta).
\end{equation}
Note that $d_p$ is also contractive, that is, $d_p(\sigma,\tau) \geq d_p(\rho+\sigma,\rho+\tau)$ for all $\sigma,\tau,\rho \in \DD(X,A)$.
\begin{defn}
The \emph{space of $p$-persistence diagrams} on the pair $(X,A)$, denoted by $\DD_p(X,A)$, is the set of all $\sigma \in \DD(X,A)$ such that $d_p(\sigma,\sigma_\varnothing) < \infty$.
\end{defn}


\begin{lem}
If $\widetilde\sigma \in \widetilde{\DD}(X)$ is a finite multiset, then $\sigma \in \DD_p(X,A)$.
\end{lem}

\begin{proof}
Let $\widetilde\sigma \in\widetilde{\DD}(X)$ be a multiset of cardinality $k < \infty$. Since $A\subset X$ is non-empty, we can pick an element $a \in A$, and so there exists a multiset $k\mset{a} = \mset{a,\ldots,a} \in \widetilde{\DD}(A)$ of cardinality $k$. Therefore, there exists a bijection between the finite multisets $\widetilde\sigma$ and $k\mset{a} = \widetilde{\sigma}_\varnothing+k\mset{a}$, implying that $d_p(\sigma,\sigma_\varnothing) \leq \widetilde{d}_p(\widetilde\sigma,\widetilde{\sigma}_\varnothing+k\mset{a}) < \infty$.
\end{proof}

\begin{lem}\label{lem:pair}
The following assertions hold:
\begin{enumerate}
\item  \label{it:extendedmetric} 
If $p=\infty$, then the function $d_p$ is an extended pseudometric on $\DD(X,A)$ and a pseudometric on $\DD_p(X,A)$.

\item \label{it:metric} If $p<\infty$, then the function $d_p$ is an extended metric on $\DD(X,A)$ and a metric on $\DD_p(X,A)$.

\end{enumerate}
\end{lem}

\begin{proof}
We will first show that $d_p$, $1\leq p\leq \infty$, is an extended pseudometric. We will then show that, for $p<\infty$, the function $d_p$ is an extended metric.

It is clear that, for all $p\in [1,\infty]$, the function  $d_p$ is symmetric, non-negative, and  $d_p(\sigma,\sigma)=0$ for all $\sigma\in \DD(X,A)$.  
The triangle inequality follows from the facts that $\widetilde{\DD}(X)$ is commutative and that $\widetilde{d}_p$ is contractive. More precisely, fix $\widetilde\rho,\widetilde\sigma,\widetilde\tau \in \widetilde{\DD}(X)$, and let $\varepsilon > 0$. By the definition of $d_p$, there exist $\widetilde\alpha, \widetilde\beta, \widetilde\gamma, \widetilde\delta \in \widetilde{\DD}(A)$ such that $\widetilde{d}_p(\widetilde\rho+\widetilde\alpha,\widetilde\sigma+\widetilde\beta) \leq d_p(\rho,\sigma)+\varepsilon$ and $\widetilde{d}_p(\widetilde\sigma+\widetilde\gamma,\widetilde\tau+\widetilde\delta) \leq  d_p(\sigma,\tau)+\varepsilon$. Using the commutativity of $\widetilde{\DD}(X)$, the contractivity of $\widetilde{d}_p$, and the triangle inequality for $\widetilde{d}_p$, we get
\begin{align*}
d_p(\rho,\tau) &\leq \widetilde{d}_p(\widetilde\rho+\widetilde\alpha+\widetilde\gamma,\widetilde\tau+\widetilde\beta+\widetilde\delta)
\\ &\leq \widetilde{d}_p(\widetilde\rho+\widetilde\alpha+\widetilde\gamma,\widetilde\sigma+\widetilde\beta+\widetilde\gamma) + \widetilde{d}_p(\widetilde\sigma+\widetilde\beta+\widetilde\gamma,\widetilde\tau+\widetilde\beta+\widetilde\delta)
\\ &\leq \widetilde{d}_p(\widetilde\rho+\widetilde\alpha,\widetilde\sigma+\widetilde\beta) + \widetilde{d}_p(\widetilde\sigma+\widetilde\gamma,\widetilde\tau+\widetilde\delta)
\\ & \leq d_p(\rho,\sigma) + d_p(\sigma,\tau) + 2\varepsilon.
\end{align*}
Our choice of $\varepsilon > 0$ was arbitrary, implying that $d_p(\rho,\tau) \leq d_p(\rho,\sigma) + d_p(\sigma,\tau)$, as required. Hence, $d_p$ is an extended pseudometric on $\DD(X,A)$. By the triangle inequality, $d_p$ is a pseudometric on $\DD_p(X,A)$. Indeed, if $\sigma,\tau \in \DD_p(X,A)$, then $d_p(\sigma,\tau) \leq d_p(\sigma,\sigma_\varnothing) + d_p(\tau,\sigma_\varnothing) < \infty$. This completes the proof of part \eqref{it:extendedmetric}.

Now, we prove part \eqref{it:metric}. Fix $p<\infty$ and let $\widetilde\sigma,\widetilde\tau \in \widetilde{\DD}(X)$ be multisets such that $\sigma\neq\tau$. It then follows that there exists a point $u \in X \setminus A$ which appears in $\widetilde\sigma$ and $\widetilde\tau$ with different multiplicities (which includes the case when it has multiplicity 0 in one of the diagrams and positive multiplicity in the other). Without loss of generality, suppose that $u$ appears with higher multiplicity in $\widetilde\sigma$. Now let $\varepsilon_1 = \inf \{ d(u,v) : v \in \widetilde\tau, v \neq u \}$. Observe that $\varepsilon_1>0$ since, otherwise, there would be a sequence of points in $\widetilde\tau$ converging to $u$ in $X$, which in turn would imply that $d_p(\tau,\sigma_\vn)=\infty$. Let $\varepsilon_2 > 0$ be such that $d(u,a) \geq \varepsilon_2$ for all $a \in A$, which exists since $u \in X \setminus A$ and $X \setminus A$ is open in $X$. We set $\varepsilon = \min\{\varepsilon_1,\varepsilon_2\}$. Now, for any $\widetilde\alpha,\widetilde\beta \in \widetilde{\DD}(A)$, if $\phi\colon \widetilde\sigma+\widetilde\alpha \to \widetilde\tau+\widetilde\beta$ is a bijection, then $\phi$ must map some copy of $u \in \widetilde\sigma$ to a point $v \in \widetilde\tau+\widetilde\beta$ with $d(u,v) \geq \varepsilon$, implying that $\widetilde{d}_p(\widetilde\sigma+\widetilde\alpha,\widetilde\tau+\widetilde\beta) \geq \varepsilon$. By taking the infimum over all $\widetilde\alpha,\widetilde\beta \in \widetilde{\DD}(A)$, it follows that $d_p(\sigma,\tau) \geq \varepsilon > 0$, as required. This shows that $d_p$ is an extended metric on $\DD(X,A)$. The triangle inequality implies, as in part \eqref{it:extendedmetric}, that $d_p$ is a metric on $\DD_p(X,A)$. This completes the proof of part \eqref{it:metric}.
\end{proof}

For $p<\infty$, the metric $d_p$ is the \textit{$p$-Wasserstein metric}. The following example shows that, for $p=\infty$, the function $d_p$ is not a metric, only a pseudometric.


\begin{example}
Let $(X,A)$ be a metric pair such that there exists a sequence $\{x_n\}_{n\in\mathbb{N}}$ of different points which converges to some $x_\infty\in X\setminus A$ and $x_\infty\neq x_n$ for all $n\in\mathbb{N}$. Then the multisets $\widetilde\sigma = \mset{x_n:n\in\NN}$ and $\widetilde\tau = \mset{x_n:n\in\NN}\cup\{x_\infty\}$ induce diagrams $\sigma,\tau\in\DD_\infty(X,A)$ such that $\sigma \neq \tau$ and $d_\infty(\sigma,\tau)=0$ as can be seen considering the sequence of bijections $\phi_n\colon\widetilde\sigma\to \widetilde\tau$ given by 
\[
\phi_n(x_i) = \begin{cases}
x_i & \text{if } i< n\\
x_{i-1} & \text{if } i > n\\
x_\infty & \text{if } i = n
\end{cases}.
\]
Thus $\DD_\infty(X,A)$ is a pseudometric space but not a metric space.
\end{example}

From now on, unless stated otherwise, we will only consider metric pairs $(X,A)$ where $X$ is a metric space. Also, for the sake of simplicity, we will treat elements in $\DD_p(X,A)$ as multisets, with the understanding that whenever we do so we are actually dealing with representatives of such elements in $\widetilde \DD(X)$. Thus, for instance, we will consider things like $x\in \sigma$ for $\sigma\in \DD_p(X,A)$ or bijections $\phi\colon \sigma\to\tau$ for $\sigma,\tau\in \DD_p (X,A)$, meaning there are representatives $\widetilde\sigma$ and $\widetilde\tau$ and a bijection $\widetilde\phi\colon \widetilde\sigma\to \widetilde\tau$. We point out that the constructions discussed above can be carried out for extended pseudometric spaces with straightforward adjustments.

Given two metric pairs $(X,A)$ and $(Y,B)$, their disjoint union is the space $(X\sqcup Y, A\sqcup B)$. We can form the extended pseudometric space 
$(X\sqcup Y, d_{X\sqcup Y})$, 
where $d_{X\sqcup Y}|_{(X\times X)}= d_X$, $d_{X\sqcup Y}|_{(Y\times Y)}= d_Y$ and  $d_{X\sqcup Y}(x,y)= \infty$ for all $x\in X$ and $y\in Y$.
The following result is an immediate consequence of the definition of the space $(\DD_p(X,A),d_p)$. 

\begin{prop}
If $(X,A)$ and $(Y,B)$ are metric pairs, then
\[
    \DD_p(X\sqcup Y,A\sqcup B) = \DD_p(X,A)\times_p \DD_p(Y,B),
\]
     where $U\times_p V$ denotes the space $U\times V$ endowed with the metric 
     \[
     d_U\times_p d_V((u_1,v_1),(u_2,v_2)) = \left(d_U(u_1,u_2)^p + d_V(v_1,v_2)^p\right)^{1/p}
     \]
if $p < \infty$, and 
     \[
     d_U\times_p d_V((u_1,v_1),(u_2,v_2)) = \max\{d_U(u_1,u_2), d_V(v_1,v_2)\}
     \]
if $p = \infty$.
\end{prop}

\begin{proof}
It is clear that for any $\sigma\in \DD_p(X\sqcup Y, A\sqcup B)$, we can write $\sigma = \sigma_{(X,A)}+\sigma_{(Y,B)}$ with $\sigma_{(X,A)}\in \DD_p(X,A)$ and $\sigma_{(Y,B)}\in \DD_p(Y,B)$. Therefore, given $\sigma,\tau\in \DD_p(X\sqcup Y,A\sqcup B)$, we have
\begin{align*}
d_p(\sigma,\tau)^p &= d_p(\sigma_{(X,A)},\tau_{(X,A)})^p+d_p(\sigma_{(Y,B)},\tau_{(Y,B)})^p\\
&= d_p\times_p d_p((\sigma_{(X,A)},\sigma_{(Y,B)}),(\tau_{(X,A)},\tau_{(Y,B)}))^p
\end{align*}
if $p < \infty$, and
\[
d_p(\sigma,\tau) = \max\{d_p(\sigma_{(X,A)},\tau_{(X,A)}), d_p(\sigma_{(Y,B)},\tau_{(Y,B)})\}
\]
if $p = \infty$.
\end{proof}

\begin{rem}
Note that, if we allow $(X,A)$ and $(Y,B)$ to be extended metric pairs, then the disjoint union $(X\sqcup Y,A\sqcup B)$ with the metric $d_{X\sqcup Y}$ defines a coproduct in the category $\mathsf{\overline{Met}_{Pair}}$ of extended metric pairs whose objects are extended metric pairs  and whose morphisms are relative Lipschitz maps (cf.\ Definition~\ref{d:metric pairs}).
\end{rem}


\begin{defn}
Given a metric pair $(X,A)$, and a relative map $f\colon (X,A)\to (Y,B)$ (i.e.\ such that $f(A)\subset B$), we define a pointed map $f_*\colon (\DD_p(X,A),\sigma_\varnothing)\to (\DD_p(Y,B),\sigma_\varnothing)$ as follows. Given a persistence diagram $\sigma\in\DD_p(X,A)$,  we let
	\begin{align}
	\label{EQ:INDUCED_MAP}
	f_*(\sigma) = \mset{f(x) : x\in\sigma}.
	\end{align}
\end{defn}

We now define the functor $\DD_p$, which we will study in the remaining sections.

\begin{prop}
\label{P:D_p_IS_FUNCTOR} Consider the map $\DD_p\colon (X,A)\mapsto (\DD_p(X,A),\sigma_\varnothing)$.
\begin{enumerate}
    \item  If $p=\infty$, then $\DD_p$  is a functor from the category $\mathsf{Met_{Pair}}$ of metric pairs equipped with relative Lipschitz maps to the category $\mathsf{PMet}_*$ of pointed pseudometric spaces with pointed Lipschitz maps.
    \item If $p<\infty$, then $\DD_p$  is a functor from the category $\mathsf{Met_{Pair}}$ of metric pairs equipped with relative Lipschitz maps to the category $\mathsf{Met}_*$ of pointed metric spaces with pointed Lipschitz maps.
\end{enumerate}

\end{prop}

\begin{proof}
Consider a $C$-Lipschitz relative map $f\colon (X,A)\to (Y,B)$, i.e. $d_Y(f(x),f(y))\leq C d_X(x,y)$ holds for all $x,y\in X$ for some $C>0$. We will prove that the pointed map
$f_*$, defined in \eqref{EQ:INDUCED_MAP}, restricts to a $C$-Lipschitz map
$\DD_p(X,A) \to \DD_p(Y,B)$.

First, given $\sigma\in \DD_p(X,A)$ a $p$-diagram, we need to prove that $f_*(\sigma)\in \DD_p(Y,B)$. For any $\sigma$, we have
\[
d_p(f_*(\sigma),\sigma_\vn)^p = \sum_{x\in {\sigma}} d_Y(f(x),B)^p \leq \sum_{x\in {\sigma}}d_Y(f(x),f(a_x))^p\leq C^p\sum_{x\in {\sigma}}d_X(x,a_x)^p
\]
for any choice $\{a_x\}_{x\in \sigma}\subset A$. Since this choice is arbitrary, 
\[
d_p(f_*(\sigma),\sigma_\vn)^p \leq C^p\sum_{x\in \sigma}d_X(x,A)^p =C^pd_p(\sigma,\sigma_\vn)< \infty.
\]

Now consider two diagrams $\sigma,\sigma'\in\DD_p(X,A)$. Observe that, if $\phi\colon \sigma\to \sigma'$ is a bijection, then it induces a bijection $f_*\phi\colon f_*(\sigma)\to f_*(\sigma')$ given by $f_*\phi(y)=f(\phi(x))$ whenever $y=f(x)$ for some $x\in \sigma$. Therefore
\[
d_p(f_*(\sigma),f_*(\sigma'))^p \leq \sum_{y\in {f_*(\sigma)}} d(y,f_*\phi(y))^p = \sum_{x\in \sigma} d(f(x),f(\phi(x)))^p\leq C^p\sum_{x\in\sigma} d(x,\phi(x))^p.
\]
Since $\phi\colon \sigma\to \sigma'$ is an arbitrary bijection, we get that
\[
d_p(f_*(\sigma),f_*(\sigma'))\leq C d_p(\sigma,\sigma').
\] Thus, $f_*\colon\DD_p(X,A)\to \DD_p(Y,B)$ is $C$-Lipschitz.

Now consider two relative Lipschitz maps $f\colon (X,A)\to (Y,B)$ and $g\colon (Y,B)\to (Z,C)$. Let $\sigma\in\DD_p(X,A)$. Then 
\begin{align*}
(g\circ f)_*(\sigma)=\mset{g\circ f(x) : x\in\sigma}=g_*(\mset{f(x) : x\in\sigma})=g_*\circ f_*(\sigma).
\end{align*}
Thus, $(g\circ f)_*=g_*\circ f_*$.

Finally, if $\Id\colon (X,A)\to (X,A)$ is the identity map, it is clear that $\Id_*\colon \DD_p(X,A)\to\DD_p(X,A)$ is also the identity map. Thus, $\DD_p$ defines a functor.
\end{proof}


\begin{rem}
Note that we could have proved that $\DD_p$ defines a functor on the category of metric spaces equipped with isometries or even bi-Lipschitz maps. However, Proposition~\ref{P:D_p_IS_FUNCTOR} is  more general.
\end{rem}

\begin{rem}
Proposition~\ref{P:D_p_IS_FUNCTOR} implies that, if $(X,A)$ is a metric pair and $(g,x)\mapsto g\cdot x$ is an action of a group $G$ on $(X,A)$ via relative bi-Lipschitz maps, then we get an action of $G$ on $\DD_p(X,A)$ given by
\[
g\cdot \sigma= \mset{g\cdot a:a\in\sigma}.
\] 
Observe that the Lipschitz constant of the bi-Lipschitz maps in the group action is preserved by the functor $\DD_p$. Hence, if $G$ acts by relative isometries on $(X,A)$ (i.e., by isometries $f\colon X\to X$ such that $f(A)\subseteq A$) then so does the induced action on $\DD_p(X,A)$. 
\end{rem}


\begin{rem}
We point out that  $\DD_p$ is, in fact, a functor from $\MetPair$ to $\monmet$, the category of commutative pointed metric monoids (see \cite{bubenik2}).
In this case, given a map $f\colon(X,A)\to(Y,B)$, the induced map $f_*\colon\DD_p(X,A)\to \DD_p(Y,B)$ is a monoid homomorphism. Composing the functor $\DD_p$ with the forgetful functor one obtains the map to $\MetPair$. In this work we consider this last composition, since we are mainly focused on the metric properties of the spaces $\DD_p(X,A)$,
and leave the study of the algebraic properties of the monoids $\DD_p(X,A)$ for future work.
\end{rem}

Consider now the quotient metric space $X/A$, namely, the quotient space induced by the partition $\{\{x\}:x\in X\setminus A\}\sqcup \{A\}$ endowed with the metric given by
\[
d([x],[y]) = \min\{d(x,y),d(x,A)+d(y,A)\}
\]
for any $x,y\in X$ (cf.\ \cite[Ch.~2, \S 22]{mun} and \cite[Definition 3.1.12]{BBI}). It follows from \cite[Remark 4.14 and Lemma 4.24]{bubenik1} that $\DD_p(X,A)$ and $\DD_p(X/A,[A])$  are isometrically isomorphic. We have the following commutative diagrams of functors. For $p=\infty$,
\begin{equation}
\label{EQ:COMM_DIAG_funct_infty}
\begin{tikzcd}
\mathsf{Met_{Pair}} \arrow{r}{\DD_\infty} \arrow[swap]{d}{\QQ} & \mathsf{PMet}_* \arrow{d}{\cong} \\
\mathsf{Met}_* \arrow{r}{\DD_\infty} & \mathsf{PMet}_*
\end{tikzcd}
\end{equation}
and, for $p<\infty$,
\begin{equation}
\label{EQ:COMM_DIAG_funct}
\begin{tikzcd}
\mathsf{Met_{Pair}} \arrow{r}{\DD_p} \arrow[swap]{d}{\QQ} & \mathsf{Met}_* \arrow{d}{\cong} \\
\mathsf{Met}_* \arrow{r}{\DD_p} & \mathsf{Met}_*
\end{tikzcd}
\end{equation}
both given by
\begin{equation*}
\begin{tikzcd}
(X,A) \arrow[maps to]{r}{} \arrow[swap,maps to]{d}{} & (\DD_p(X,A),\sigma_\varnothing) \arrow[maps to]{d}{} \\
(X/A,[A]) \arrow[maps to]{r}{} & (\DD_p(X/A,[A]),\sigma_\varnothing) 
\end{tikzcd}.
\end{equation*}
Observe that the map $\DD_p(X,A)\mapsto \DD_p(X/A,[A])$ is a natural isomorphism. Therefore,  diagrams~\eqref{EQ:COMM_DIAG_funct} and \eqref{EQ:COMM_DIAG_funct_infty} show that the functor $\DD_p$ factors through the quotient functor $\QQ\colon (X,A)\mapsto (X/A,[A])$ and the functor $(X/A,[A])\mapsto \DD_p(X/A,[A])$ for $p\in [1,\infty]$.


\begin{rem}
Note that we also have the following commutative diagrams of functors. For $p=\infty$,
\begin{equation}
\begin{tikzcd}
\mathsf{Met_{Pair}} \arrow{r}{} \arrow[swap]{d}{\QQ} & \mathsf{\overline{PMet}_{Pair}} \arrow{d}{\cong} \\
\mathsf{Met}_* \arrow{r}{} & \mathsf{\overline{PMet}}_*
\end{tikzcd}
\end{equation}
and, for $p<\infty$,
\begin{equation}
\begin{tikzcd}
\mathsf{Met_{Pair}} \arrow{r}{} \arrow[swap]{d}{\QQ} & \mathsf{\overline{Met}_{Pair}} \arrow{d}{\cong} \\
\mathsf{Met}_* \arrow{r}{} & \mathsf{\overline{Met}}_*
\end{tikzcd}
\end{equation}
both given by
\begin{equation*}
\begin{tikzcd}
(X,A) \arrow[maps to]{r}{} \arrow[swap,maps to]{d}{} & (\widetilde{\DD}(X),\widetilde{\DD}(A),\widetilde{d}_p) \arrow[maps to]{d}{} \\
(X/A,[A]) \arrow[maps to]{r}{} & (\DD(X/A,[A]),\sigma_\varnothing,d_p)
\end{tikzcd}.
\end{equation*}
Here the categories $\mathsf{\overline{PMet}_{Pair}}$, $\mathsf{\overline{PMet}_{*}}$, $\mathsf{\overline{Met}_{Pair}}$ and $\mathsf{\overline{Met}_{*}}$ consist of extended (pseudo)metric pairs and pointed (pseudo)metric spaces respectively. 
\end{rem}
\begin{rem}
Observe that the subspace of $\DD_p(X,A) \cong \DD_p(X/A,[A])$ consisting of diagrams with finitely many points can be identified, as a set, with the infinite symmetric product of the pointed space $(X/A,[A])$ (see \cite[p. 282]{Hatcher} for the relevant definitions). These two spaces, however, might not be homeomorphic in general, as the infinite symmetric product is not metrizable unless $A$ is open in $X$ (see, for instance, \cite{wofsey}). 
\end{rem}


\subsection{Alexandrov spaces} 
Let $X$ be a metric space. The \emph{length} of a continuous path $\xi\colon [a,b] \to X$ is given by
\[
L(\xi) = \sup \left\lbrace  \sum_{i=0}^{n-1}d(\xi(t_i),\xi(t_{i+1}))  \right\rbrace,
\] where the supremum is taken over all finite partitions $a=t_0\leq t_1\leq \dots \leq t_n =b$ of the interval $[a,b]$. A \emph{geodesic space} is a metric space $X$ where for any $x,y\in X$ there is a \emph{shortest path} (or \emph{minimizing geodesic}) between $x, y\in X$, i.e. a path $\xi$ such that
\begin{equation}\label{esp-long}
d(x,y) =  L(\xi).
\end{equation}
In general, a path $\xi\colon J \to X$, where $J$ is an interval, is said to be \emph{geodesic} if each $t\in J$ has a neighborhood $U \subset J$ such that $\xi|_U$ is a shortest path between any two of its points. 

We will also consider the \emph{model spaces} $\mathbb{M}^n_k$ given by
\[
\mathbb{M}^n_\kappa = \left\lbrace
\begin{array}{ll}
\mathbb{S}^n\left(\frac{1}{\sqrt{\kappa}}\right),  &  \mbox{if } \kappa>0,\\
\mathbb{R}^n,  &  \mbox{if } \kappa=0,\\
\mathbb{H}^n\left(\frac{1}{\sqrt{-\kappa}}\right),  &  \mbox{if } \kappa<0.
\end{array}
\right.
\]

\begin{defn}
\label{def:triangulos}
A \emph{geodesic triangle} $\triangle pqr$ in $X$ consists of three points $p,q,r\in X$ and three minimizing geodesics $[pq],\ [qr],\ [rp]$ between those points. A \emph{comparison triangle} for $\triangle pqr$ in $\mathbb{M}^2_k$ is a geodesic triangle $\widetilde{\triangle}_k pqr = \triangle \widetilde{p}\widetilde{q}\widetilde{r}$ in $\mathbb{M}^2_k$ such that
\[
d(\widetilde{p},\widetilde{q})=d(p,q),\ d(\widetilde{q},\widetilde{r})=d(q,r),\ d(\widetilde{r},\widetilde{p})=d(r,p).
\]
\end{defn}
\begin{defn}\label{def:alex}
We say that $X$ is an \emph{Alexandrov space with curvature bounded below by $k$} if $X$ is complete, geodesic and can be covered with open sets with the following property (cf.~Figure~\ref{fig:alex}):
\begin{itemize}
    \myitem[(T)] For any geodesic triangle $\triangle pqr$ contained in one of these open sets, any comparison triangle $\widetilde{\triangle}_k pqr$ in $\mathbb{M}^2_k$ and any point $x\in [qr]$, the corresponding point $\widetilde{x}\in [\widetilde{q}\widetilde{r}]$ such that $d(\widetilde{q},\widetilde{x})=d(q,x)$ satisfies
\begin{align*}
    d(p,x)\geq d(\widetilde{p},\widetilde{x}).
\end{align*}
\label{IT:PROPERTY_T}
\end{itemize}
\end{defn}

\begin{figure}
\centering
\begin{tikzpicture}
\draw [very thick] (0,0) node[below left] {$r$} to[bend left=20] (1,3) node[above] {$p$} to[bend left=20] (4,0) node[below right] {$q$} to[bend left=20] node[pos=0.6] (d) {} cycle;
\draw (1,3) to[bend left=10] (d.center) node[below] {$x$};
\draw (2,-1.3) node {$X$};
\end{tikzpicture}
\qquad\qquad\qquad
\begin{tikzpicture}
\draw [very thick] (-0.2,-0.2) node[below left] {$\widetilde{r}$} to (1,3) node[above] {$\widetilde{p}$} to (4.2,-0.2) node[below right] {$\widetilde{q}$} to node[pos=0.6] (d) {} cycle;
\draw (1,3) to (d.center) node[below] {$\widetilde{x}$};
\draw (2,-1.3) node {$\mathbb{M}^2_\kappa$};
\end{tikzpicture}
\caption{The condition for a complete geodesic metric space $X$ to be an Alexandrov space with curvature $\geq \kappa$. Here, the curves $[pq]$, $[qr]$, $[rp]$, $[px]$, $[\widetilde{p}\widetilde{q}]$, $[\widetilde{q}\widetilde{r}]$, $[\widetilde{r}\widetilde{p}]$, $[\widetilde{p}\widetilde{x}]$ are geodesics, and the length of $[pq]$ (respectively, $[rp]$, $[qx]$, $[xr]$) is equal to the length of $[\widetilde{p}\widetilde{q}]$ (respectively, $[\widetilde{r}\widetilde{p}]$, $[\widetilde{q}\widetilde{x}]$, $[\widetilde{x}\widetilde{r}]$). Condition (T) then says that the length of $[\widetilde{p}\widetilde{x}]$ is not greater than the length of $[px]$.}
\label{fig:alex}
\end{figure}

By Toponogov's Globalization Theorem, if $X$ is an Alexandrov space with curvature bounded below by $k$, then property~\ref{IT:PROPERTY_T} above holds for any geodesic triangle in $X$ (see, for example,  \cite[Section 3.4]{Plaut}). Well-known examples of Alexandrov spaces include complete Riemannian $n$-manifolds with a uniform lower sectional curvature bound, orbit spaces of such manifolds by an effective, isometric action of a compact Lie group, and, in infinite dimension, Hilbert spaces. The latter are instances of infinite-dimensional Alexandrov spaces of non-negative curvature.

The \emph{angle} between two minimizing geodesics $[pq]$, $[pr]$ in an Alexandrov space $X$ is defined as
\[
\angle qpr =\lim_{q_1,r_1\to p}\{\angle\widetilde{q}_1\widetilde{p}\widetilde{r}_1 : q_1\in [pq],\ r_1\in [pr] \}.
\]
Geodesics that make an angle zero determine an equivalence class called \emph{tangent direction}. The set of tangent directions at a point $p\in X$ is denoted by $\Sigma'_p$. When equipped with the angle  distance $\angle$, the set $\Sigma_p'$ is a metric space. Note that the metric space $(\Sigma'_p,\angle)$ may fail to be complete, as one can see by considering directions at a point in the boundary of the unit disc $D$ in the Euclidean plane, $D$ being an Alexandrov space of non-negative curvature. The completion of  $\left(\Sigma_p', \angle\right)$  is called the \emph{space of directions of $X$ at $p$}, and is denoted by $\Sigma_p$. Note that in a complete, finite-dimensional Riemannian manifold $M^n$ with sectional curvature uniformly bounded below, the space of directions at any point is isometric to the unit sphere in the tangent space to the manifold at the given point. For further basic results on Alexandrov geometry, we refer the reader to \cite{BBI,BGP,Plaut}.

We conclude this section by briefly recalling the definition of the Hausdorff dimension of a metric space (see \cite[Section 1.7]{BBI} for further details). One may show that the Hausdorff dimension of an Alexandrov space is an integer or infinite (see \cite[Corollary 10.8.21 and Exercise 10.8.22]{BBI}).

Let $X$ be a metric space and denote the diameter of a subset $S\subset X$ by $\diam(S)$. For any $\delta\in [0, \infty)$ and any $\varepsilon>0$, let
\begin{align*}
\mathcal{H}^\delta_\varepsilon (X) = \inf \left\{ \sum_{i\in\NN} (\diam(S_i))^\delta :
X\subset \bigcup_{i\in\NN} S_i \text{ and } \diam(S_i) < \varepsilon\right\},
\end{align*}
where $\{S_i\}_{i\in\NN}$ is a countable covering of $X$ by sets of diameter less than $\varepsilon$. Note that if no such covering exists, then $\mathcal{H}^\delta_\varepsilon (X) = \infty$. The \textit{$\delta$-dimensional Hausdorff measure of $X$} is given by
\[
\mathcal{H}^\delta (X) = \omega_\delta\cdot \lim_{\varepsilon\searrow 0} \mathcal{H}^\delta_\varepsilon (X),
\]
where $\omega_\delta>0$ is a normalization constant such that, if $\delta$ is an integer $n$, the $n$-dimensional Hausdorff measure of the unit cube in $n$-dimensional Euclidean space $\R^n$ is $1$. This is achieved by letting $\omega_n$ be the Lebesgue measure of the unit ball in $\R^n$. As its name indicates, the Hausdorff measure is a measure on the Borel $\sigma$-algebra of $X$.
One may show that there exists $0\leq \delta_o\leq \infty$ such that $\mathcal{H}^\delta(X) = 0$ for all $\delta>\delta_o$ and $\mathcal{H}^\delta(X) =\infty$ for all $\delta<\delta_o$. We then define the \textit{Hausdorff dimension of $X$}, denoted by $\dim_H(X)$, to be  $\delta_o$. Thus,
\begin{align*}
\dim_H(A) & = \sup \{\, \delta : \mathcal{H}^\delta (X)> 0\,\} 
            = \sup \{\, \delta : \mathcal{H}^\delta (X) = \infty\,\}\\
        &   = \inf \{\, \delta : \mathcal{H}^\delta(X) = 0\,\} 
            = \inf \{\, \delta : \mathcal{H}^\delta(X) < \infty\,\} .
\end{align*}


\section{Gromov--Hausdorff convergence and sequential continuity}
\label{S:CONTINUITY}

In this section, we investigate the continuity of the functors $\mathcal{Q}$ and $\DD_p$ defined in the preceding section. Since $\DD_p$, $p<\infty$, takes values in $\Met_*$, the category of pointed metric spaces, while $\DD_\infty$ takes values in $\mathsf{PMet}_{*}$, the category of pointed pseudometric spaces, we will consider each case separately. As both $\mathcal{Q}$ and $\DD_p$ are defined on $\MetPair$, the category of metric pairs, we will first define a notion of Gromov--Hausdorff convergence of metric pairs $(X,A)$. We do this in such a way that when $A$ is a point, our definition implies the usual pointed Gromov--Hausdorff convergence of pointed metric spaces (see~\cite[Definition 8.1.1]{BBI} and \cite{herron}; cf.~\cite[Definition 2.1]{DJ}  for the case of proper metric spaces; see also Definition~\ref{d:pseudopairsGHconv} below). 
After showing that $\mathcal{Q}\colon \mathsf{Met_{Pair}}  \to \mathsf{Met_*}$ is sequentially continuous with respect to the Gromov--Hausdorff convergence of metric pairs, and that $\DD_p\colon \mathsf{Met_{Pair}}  \to \mathsf{Met_*}$, $p<\infty$, is not always sequentially continuous, we will prove the sequential continuity of $\DD_\infty\colon \MetPair\to\mathsf{PMet}_*$ with respect to the Gromov--Hausdorff convergence of metric pairs on $\MetPair$ and pointed Gromov--Hausdorff convergence of pseudometric spaces on $\mathsf{PMet}_*$.

\begin{defn}[Gromov--Hausdorff convergence for metric pairs]
\label{d:pairsGHconv}
A sequence $\{(X_i,A_i)\}_{i\in\NN}$ of metric pairs \emph{converges in the Gromov--Hausdorff topology  to a metric pair} $(X,A)$ if there exist sequences $\{\varepsilon_i\}_{i\in\NN}$ and $\{R_i\}_{i\in\NN}$ of positive numbers with $\varepsilon_i\searrow 0$, $R_i\nearrow \infty$, and $\varepsilon_i$-approximations from $\overline{B}_{R_i}(A_i)$ to $\overline{B}_{R_i}(A)$ for each $i\in \mathbb{N}$, i.e.\ maps ${f}_i\colon \overline{B}_{R_i}(A_i)\to X$ satisfying the following three conditions:
\begin{enumerate}
    \item\label{d:pairsGHconv1} $|d_{X_i}(x,y)-d_X({f}_i(x),{f}_i(y)|\leq\varepsilon_i$ for any $x,y\in \overline{B}_{R_i}(A_i)$;
    \item\label{d:pairsGHconv2} $d_H({f}_i(A_i),A)\leq\varepsilon_i$, where $d_H$ stands for the Hausdorff distance in $X$;
    \item\label{d:pairsGHconv3} $\overline{B}_{R_i}(A)\subset \overline{B}_{\varepsilon_i}({f}_i(\overline{B}_{R_i}(A_i)))$. 
\end{enumerate}
We will denote the Gromov--Hausdorff convergence of metric pairs by $(X_i,A_i)
\GHtopair(X,A)
$ and the pointed Gromov--Hausdorff convergence by $(X_i,x_i)\GHtopointed(X,x)$.
\end{defn} 

With Definition~\ref{d:pairsGHconv} in hand, we now show that the functor $\QQ$ is continuous, while $\DD_p$ is not necessarily continuous when $p<\infty$.


\begin{prop}\label{prop:continuity of quotient functor}
The quotient functor $\mathcal{Q}\colon \mathsf{Met_{Pair}}  \to \mathsf{Met_*}$, given by $(X,A)  \mapsto (X/A,[A])$, is sequentially continuous with respect to the Gromov--Hausdorff convergence of metric pairs.
\end{prop}

\begin{proof}
We will prove that, if there exist sequences $\{\varepsilon_i\}_{i\in\NN}$ and $\{R_i\}_{i\in\NN}$ of positive numbers with $\varepsilon_i\searrow 0$, $R_i\nearrow \infty$ and $\varepsilon_i$-approximations from $\overline{B}_{R_i}(A_i)$ to $\overline{B}_{R_i}(A)$, then there exist $(5\varepsilon_i)$-approximations from $\overline{B}_{R_i}([A_i])\subset (X_i/A_i,[A_i])$ to $\overline{B}_{R_i}([A])\subset (X/A,[A])$. For ease of notation, we will omit the subindices in the metric which indicate the corresponding metric space.

Let ${f}_i$ be an $\varepsilon_i$-approximation from $\overline{B}_{R_i}(A_i)$ to $\overline{B}_{R_i}(A)$ in the sense of Definition \ref{d:pairsGHconv}. Then, for any $x\in \overline{B}_{R_i}(A_i)$, $a_i \in A_i$, we have
\[
   |d(x,a_i) - d({f}_i(x),{f}_i(a_i))|\leq\varepsilon_i 
\]
which implies 
\begin{equation}\label{eq:quotient-continuity1}
   |d(x,A_i) - d({f}_i(x),{f}_i(A_i))|\leq \varepsilon_i. 
\end{equation}
Moreover, for any $a_i\in A_i$ and $a\in A$, we have
\[
  |d({f}_i(x),{f}_i(a_i))-d({f}_i(x),a)|\leq d({f}_i(a_i),a) 
\]
and, since $d_H({f}_i(A_i),A)\leq \varepsilon_i$, this yields
\begin{equation}\label{eq:quotient-continuity2}
   |d({f}_i(x),{f}_i(A_i)) - d({f}_i(x),A)|\leq \varepsilon_i. 
\end{equation}
Combining inequalities \eqref{eq:quotient-continuity1} and \eqref{eq:quotient-continuity2}, we get
\[
    |d(x,A_i) - d({f}_i(x),A)|\leq 2\varepsilon_i.
\]
Now, for each $i$, define $\underline{{f}}_i\colon \overline{B}_{R_i}([A_i])\to X/A$ by
\[
\underline{{f}}_i([x])=
\begin{cases}
[{f}_i(x)] & \text{if } [x]\neq [A_i],\\
[A] & \text{if } [x] = [A_i].
\end{cases}
\]
We will prove that $\underline{{f}}$ is a $(5\varepsilon_i)$-approximation from $\overline{B}_{R_i}([A_i])$ to $\overline{B}_{R_i}([A])$. Indeed, consider $[x],[y]\in \overline{B}_{R_i}([A_i])$. Then $x,y\in B_{R_i}(A_i)$ and therefore
\begin{align*}
&|d([x],[y])-d(\underline{{f}}_i([x]),\underline{{f}}_i([y]))|\\ &= |\min\{d(x,y),d(x,A_i)+d(y,A_i)\}-\min\{d({f}_i(x),{f}_i(y)),d({f}_i(x),A)+d({f}_i(y),A)\}|\\
&\leq 
|d(x,y)-d({f}_i(x),{f}_i(y))|+|d(x,A_i)-d({f}_i(x),A)|+|d(y,A_i)-d({f}_i(y),A)|\\
&\leq \varepsilon_i+2\varepsilon_i+2\varepsilon_i = 5\varepsilon_i.
\end{align*}
If $[x]\neq[A_i]$ and $[y]=[A_i]$, then
\[
|d([x],[y])-d(\underline{{f}}_i([x]),\underline{{f}}_i([y]))| =|d(x,A_i)-d({f}_i(x),A)|\leq 2\varepsilon_i.
\] A similar inequality is obtained when $[y]\neq[A_i]$ and $[x]=[A_i]$. When both $[x]=[A_i]$ and $[y]=[A_i]$, we get
\[
|d([x],[y])-d(\underline{{f}}_i([x]),\underline{{f}}_i([y]))| = 0.
\]
In any case, we see that the distortion of $\underline{{f}}_i$ is $\leq 5\varepsilon_i$, which is item \eqref{d:pairsGHconv1} in Definition \ref{d:pairsGHconv}. 

For item \eqref{d:pairsGHconv2} in Definition \ref{d:pairsGHconv}, we simply observe that by definition of $\underline{{f}}_i$ they are pointed maps.

Finally, we see that for $[y]\in B_{R_i}([A])$ we have $d(y,A)\leq R_i$, so given that ${f}_i$ is an $\varepsilon_i$-approximation from $B_{R_i}(A_i)$ to $B_{R_i}(A)$ there exists $x\in B_{R_i}(A_i)$ such that $d(y,{f}_i(x))\leq \varepsilon_i$. Therefore, 
\[d([y],\underline{{f}}_i([x]))\leq d(y,{f}_i(x))\leq \varepsilon_i.
\]
Thus $[y]\in B_{\varepsilon_i}(\underline{{f}}_i(B_{R_i}(A_i)))$. This gives item \eqref{d:pairsGHconv3} in Definition \ref{d:pairsGHconv}.
\end{proof}

\begin{example}[$\DD_p\colon \mathsf{Met_{Pair}}\to \mathsf{Met}_*$ with $p<\infty$ is not 
sequentially continuous]
\label{R:NO_CONTINUITY_P_FINITE}
Let $X_i=[-\frac{1}{i},\frac{1}{i}]\subset \mathbb R$ and set $A_i=X=A=\{0\}$. Then $\mathcal D_p(X,A)=\{\sigma_{\varnothing}\}$. Observe that for $p\neq\infty$, the space $\mathcal D_p(X_i,A_i)$ is unbounded. Indeed, if $\sigma_n$ is the diagram that contains a single point, $1/i$, with multiplicity $n$, then $d_p(\sigma_n,\sigma_{\varnothing})=\sqrt[p]{n}/{i}\to \infty$ as $n\to\infty$. 

Now, let $\sigma_{\varnothing}^i\in \mathcal D_p(X_i,A_i)$ be the empty diagram and suppose, for the sake of contradiction, that there exist $\varepsilon_i$-approximations ${f}_i\colon \overline{B}_{R_i}(\sigma^i_\varnothing)\to \mathcal D_p(X,A)$ for some $\varepsilon_i \searrow 0$ and $R_i \nearrow \infty$. Then 
\[
    |d_p(\sigma,\sigma_{\varnothing}^i)-d_p({f}_i(\sigma),{f}_i(\sigma_{\varnothing}^i))|\leq\varepsilon_i
\] for all $\sigma\in\overline{B}_{R_i}(\sigma_{\varnothing}^i)$. However, we have $d_p({f}_i(\sigma),{f}(\sigma_{\varnothing}^i))=d_p(\sigma_{\varnothing},\sigma_{\varnothing})=0$, implying that 
\begin{align}
\label{EQ:INEQ_CONTRADICTION_NOT_CONTINUOUS}
d_p(\sigma,\sigma^i_\varnothing) \leq \varepsilon_i    
\end{align}
 for all $\sigma\in\overline{B}_{R_i}(\sigma_{\varnothing}^i)$. As $\varepsilon_i \to 0$ and $R_i \to \infty$ as $i \to \infty$, inequality \eqref{EQ:INEQ_CONTRADICTION_NOT_CONTINUOUS}
 contradicts the fact that $\mathcal{D}_p(X_i,A_i)$ is unbounded for each $i$.
\end{example}

Finally, we turn our attention to the functor $\DD_\infty$. Recall, from Section~\ref{S:PRELIMINARIES}, that $\DD_\infty$ takes values in $\mathsf{PMet}_*$, the category of pointed pseudometric spaces. Thus, to discuss the continuity of $\DD_\infty$, we must first define a notion of Gromov--Hausdorff convergence for pointed pseudometric spaces. We define this convergence in direct analogy to pointed Gromov--Hausdorff convergence of pointed metric spaces.


\begin{defn}[Gromov--Hausdorff convergence for pointed pseudometric spaces]
\label{d:pseudopairsGHconv}
A sequence $\{(D_i,\sigma_i)\}_{i\in\NN}$ of pointed pseudometric spaces \emph{converges in the Gromov--Hausdorff topology  to a pointed pseudometric space} $
(D,\sigma)$ if there exist sequences $\{\varepsilon_i\}_{i\in\NN}$ and $\{R_i\}_{i\in\NN}$ of positive numbers with $\varepsilon_i\searrow 0$, $R_i\nearrow \infty$, and $\varepsilon_i$-approximations from $\overline{B}_{R_i}(\sigma_i)$ to $\overline{B}_{R_i}(\sigma)$ for each $i\in \mathbb{N}$, i.e.\ maps ${f}_i\colon \overline{B}_{R_i}(\sigma_i)\to D$ satisfying the following three conditions:
\begin{enumerate}
    \item\label{d:pseudopairsGHconv1} $|d_{D_i}(x,y)-d_D({f}_i(x),{f}_i(y)|\leq\varepsilon_i$ for any $x,y\in \overline{B}_{R_i}(\sigma_i)$;
    \item\label{d:pseudopairsGHconv2} $d({f}_i(\sigma_i),\sigma)\leq\varepsilon_i$;
    \item\label{d:pseudopairsGHconv3} $\overline{B}_{R_i}(\sigma)\subset \overline{B}_{\varepsilon_i}({f}_i(\overline{B}_{R_i}(\sigma_i)))$. 
\end{enumerate}
As for metric spaces, we will also denote the Gromov--Hausdorff convergence of pseudometric pairs by $(D_i,\sigma_i)\GHtopointed (D,\sigma)$.
\end{defn} 


Given a pseudometric space $D$, we will denote by $\underline{D}$ the metric quotient $D/\sim$, where $c \sim d$ if and only if $d_D(c,d)=0$. We also denote sometimes by $\underline{x}$ the image of $x\in D$ under the metric quotient.
The following proposition shows that pointed Gromov--Hausdorff convergence of pseudometric spaces induces pointed Gromov--Hausdorff convergence of the corresponding metric quotients. 

\begin{prop}
Let $\{(D_i,\sigma_i)\}_{i\in \NN}$, $(D,\sigma)$ be pointed pseudometric spaces and let $\pi_i\colon D_i\to \underline{D}_i$, $\pi\colon D\to \underline{D}$ be the canonical metric identifications. Then the following assertions hold:
\begin{enumerate}
\item If $(D_i,\sigma_i)\GHtopointed (D,\sigma)$, then $(\underline{D}_i,\underline{\sigma}_i)\GHtopointed (\underline{D},\underline{\sigma})$. 
\item If $(\underline{D}_i,\underline{\sigma}_i)\GHtopointed (\underline{D},\underline{\sigma})$, then $(D_i,\sigma_i)\GHtopointed ({D},{\sigma})$.
\end{enumerate}
\end{prop}
\begin{proof}
For each $i$, consider $s_i\colon \underline{D}_i\to D_i$ such that $\pi_i(s_i(x)) = x$ for all $x\in \underline{D}_i$ and $s\colon \underline{D}\to D$ similarly. These maps exist due to the axiom of choice.
Let $f_i$ be $\eps_i$-approximations from $\overline{B}_{R_i}(\sigma_i)$ to $\overline{B}_{R_i}(\sigma)$. Define $\underline{f}_i\colon\overline{B}_{R_i}(\underline{\sigma}_i)\to \underline{D}$ as
\[
\underline{f}_i(x) = \pi(f_i(s_i(x)))
\] for any $x\in \underline{D}_i$. Then $\underline{f}_i$ is a $(2\eps_i)$-approximation from $\overline{B}_{R_i}(\underline{\sigma}_i))$ to $\overline{B}_{R_i}(\underline{\sigma})$. Indeed, 
\begin{align*}
    |d(x,y)-d(\underline{f}_i(x),\underline{f}_i(y))| &= |d(s_i(x),s_i(y))-d(f_i(s_i(x)),f_i(s_i(y)))| \leq \eps_i.
\end{align*}
Also 
\begin{align*}
d(\underline{f}_i(\underline{\sigma}_i),\underline{\sigma}) &= d(f_i(s_i(\underline{\sigma}_i)),\sigma)\\
&\leq d(f_i(s_i(\underline{\sigma}_i)),f_i(\sigma_i))+d(f_i(\sigma_i),\sigma)\\
&\leq d(s_i(\underline{\sigma}_i),\sigma_i)+\eps_i+d(f_i(\sigma_i),\sigma)\\
&\leq 2\eps_i.
\end{align*}
Moreover, if $d(x,\pi(\sigma))\leq R_i$ then $d(s(x),\sigma) \leq R_i$. Then there is some $y\in D_i$ with $d(y,\sigma_i)\leq R_i$ such that $d(s(x),f_i(y))\leq \eps_i$. Therefore,
\begin{align*}
d(x,\underline{f}_i(\underline{y})) &= d(s(x),f_i(s_i(\underline{y})))\\
&\leq d(s(x),f_i(y))+d(f_i(y),f_i(s_i(\underline{y})))\\
&\leq \eps_i + d(y,s_i(\underline{y}))+\eps_i\\
&=2\eps_i.
\end{align*}
This proves item (1).

Conversely, given $\underline{f}_i$ an $\eps_i$-approximation from $\overline{B}_{R_i}(\underline{\sigma}_i)$ to $\overline{B}_{R_i}(\underline{\sigma})$, we can define ${f}_i\colon\overline{B}_{R_i}(\sigma_i)\to {D}$ as
\[
{f}_i(x)=s(\underline{f}_i(\underline{x}))
\] for any $x\in D_i$. Then ${f}_i$ is an $\eps_i$-approximation from $\overline{B}_{R_i}(\sigma_i)$ to $\overline{B}_{R_i}(\sigma)$. Indeed, 
\[
|d(x,y)-d({f}_i(x),{f}_i(y))| = |d(\underline{x},\underline{y})-d(\underline{f}_i(\underline{x}),\underline{f}_i(\underline{y}))| \leq \eps_i.
\]
Moreover 
\[
d({f}_i(\sigma_i),\sigma)=d(\underline{f}_i(\underline{\sigma}_i),\underline{\sigma}) \leq \eps_i.
\]
Finally, if $d(\underline{x},\sigma)\leq R_i$ then there exists $y\in D_i$ such that $d(\underline{y},\underline{\sigma}_i)\leq R_i$ and $d(\underline{x},\underline{f}_i(\underline{y}))\leq \eps_i$, or equivalently, $d(x,{f}_i(y))\leq \eps_i$. This proves item (2).
\end{proof}

In particular, if we consider the following commutative diagram
\[
\begin{tikzcd}
\mathsf{Met_{Pair}} \arrow{r}{\DD_\infty} \arrow[swap]{dr}{\pi \circ \DD_\infty} & \mathsf{PMet}_* \arrow{d}{\pi } \\
& \mathsf{Met}_*
\end{tikzcd}
\]
where $\pi\colon \PMet_*\to \Met_*$ is the canonical metric identification functor, then $\DD_\infty$ is continous if and only if $\pi\circ \DD_\infty$ is continuous.

We will now show that, if $(X_i,A_i)\GHtopair (X,A)$, then $(\DD_\infty(X_i,A_i),\sigma^i_\vn)\GHtopointed (\DD_\infty(X,A),\sigma_\vn)$.


\begin{prop}
\label{prop:continuity}
The functor $(X,A)\mapsto (\DD_\infty(X,A),\sigma_{\varnothing})$ is sequentially continuous with respect to the Gromov--Hausdorff convergence of metric pairs.
\end{prop}

\begin{proof}
Let $(X_i,A_i)\GHto (X,A)$, $R_i\nearrow \infty$, $\varepsilon_i\searrow 0$, and ${f}_i$ be $\varepsilon_i$-approximations from $\overline{B}_{R_i}(A_i)$ to $\overline{B}_{R_i}(A)$. We can define a map $(f_i)_*\colon\overline{B}_{R_i}(\sigma^i_\vn)\to \DD_\infty(X,A)$ as
\[
(f_i)_*(\sigma) = \mset{{f}_i(x) : x\in \sigma\setminus A_i}.
\]
We will prove that $(f_i)_*$ is a $(3\varepsilon_i)$-approximation from $\overline{B}_{R_i}(\sigma^i_\vn)$ to $\overline{B}_{R_i}(\sigma_\vn)$.

Let $\sigma,\sigma'\in\DD_\infty(X_i,A_i)$. We now show that, for any bijection $\phi\colon \sigma \to \sigma'$, there exists a bijection $\phi_*\colon (f_i)_*(\sigma) \to (f_i)_*(\sigma')$ such that
\begin{align}
\label{eq:gammahat}
    \left| \sup_{x \in \sigma} d_{X_i}(x,\phi(x)) - \sup_{y \in (f_i)_*(\sigma)} d_X(y,\phi_*(y)) \right| \leq 3\varepsilon_i,
\end{align}
and, conversely, that for any bijection $\phi_*\colon (f_i)_*(\sigma) \to (f_i)_*(\sigma')$, there exists a bijection $\phi\colon\sigma \to \sigma'$ such that inequality \eqref{eq:gammahat} holds.

Indeed, let $\phi\colon \sigma \to \sigma'$ be a bijection, and let $x \in\sigma$ and $x' \in \sigma'$ be such that $\phi(x) = x'$. We set $\phi_*(\widehat{x}) = \widehat{x'}$, where, given any $z \in X_i$, we set $\widehat{z} = {f}_i(z)$ if $z \notin A_i$, and we set $\widehat{z} \in A$ to be a point such that $d_X({f}_i(z),\widehat{z}) \leq \varepsilon_i$ if $z \in A_i$. In the latter case, such a choice is possible by item \eqref{d:pairsGHconv2} in Definition \ref{d:pairsGHconv}. In particular, in either case we have $d_X({f}_i(z),\widehat{z}) \leq \varepsilon_i$. Up to changing representatives of $(f_i)_*(\sigma)$ and $(f_i)_*(\sigma')$ in $\DD_\infty(X,A)$, this completely defines a bijection $\phi_*\colon (f_i)_*(\sigma)\to (f_i)_*(\sigma')$, and we have
\begin{align*}
\left| d_{X_i}(x,x') - d_X(\widehat{x},\widehat{x'}) \right| &\leq \left| d_{X_i}(x,x') - d_X({f}_i(x),{f}_i(x')) \right| + \left| d_X({f}_i(x),{f}_i(x')) - d_X(\widehat{x},{f}_i(x')) \right| \\
&\qquad{} + \left| d_X(\widehat{x},{f}_i(x')) - d_X(\widehat{x},\widehat{x'}) \right| \\
&\leq \varepsilon_i + d_X({f}_i(x),\widehat{x}) + d_X({f}_i(x'),\widehat{x'}) \\
&\leq 3\varepsilon_i
\end{align*}
by item \eqref{d:pairsGHconv1} in Definition~\ref{d:pairsGHconv} and the triangle inequality. Taking the supremum over all $x \in\sigma$ yields inequality \eqref{eq:gammahat}.

Conversely, let $\theta\colon (f_i)_*(\sigma) \to (f_i)_*(\sigma')$ be a bijection, and let $y \in (f_i)_*(\sigma)$ and $y' \in (f_i)_*(\sigma')$ be such that $\theta(y) = y'$. We define a bijection $\breve\theta\colon \sigma \to \sigma'$ by setting $\breve\theta(\breve{y}) = \breve{y'}$, where, given any $z \in X$ (viewed as an element in the multiset $(f_i)_*(\sigma)$ or $(f_i)_*(\sigma')$), we set $\breve{z} \in X_i$ to be such that ${f}_i(\breve{z}) = z$ if $z$ is defined as ${f}_i(x)$ for some $x \in X_i$, and such that $\breve{z} \in A_i$ and $d_X({f}_i(\breve{z}),z) \leq \varepsilon_i$ otherwise. In the latter case, we must have $z \in A$ and hence such a choice is possible by item \eqref{d:pairsGHconv2} in Definition~\ref{d:pairsGHconv}. Similarly as above, we can then show that $\left| d_{X_i}(\breve{y},\breve{y'}) - d_X(y,y') \right| \leq 3\varepsilon_i$, and hence \eqref{eq:gammahat} holds with $\phi = \breve\theta$ and $\phi_* = \theta$.

Therefore, for any $\sigma,\sigma'\in\overline{B}_{R_i}(\sigma_\vn^i)$, we have
\[
|d_\infty(\sigma,\sigma')-d_\infty({f}_i(\sigma),{f}_i(\sigma'))| = \left|\inf_\phi\sup_{x\in \sigma}\{d(x,\phi(x))\}-\inf_{\theta}\sup_{y\in (f_i)_*(\sigma)}\{d(y,\theta(y))\}\right|\leq 3\varepsilon_i.
\]

On the other hand, by definition, we have that
\[
d_\infty({f}_i(\sigma^i_\vn),\sigma_\vn)=d_\infty(\sigma_\vn,\sigma_\vn)=0\leq 3\varepsilon_i.
\]

Finally, if $d_\infty(\sigma,\sigma_\vn)\leq R_i$, then $d(y,A) \leq R_i$ for any $y\in \sigma$, and since ${f}_i$ is an $\varepsilon_i$-approximation from $\overline{B}_{R_i}(A_i)$ to $\overline{B}_{R_i}(A)$, we know that there is some $x_y\in \overline{B}_{R_i}(A_i)$ such that $d(y,{f}_i(x_y))\leq\varepsilon_i$. Hence, the diagram $\hat{\sigma}\in \DD_\infty(X_i,A_i)$ given by \[
\hat\sigma = \mset{x_y:x\in\sigma}
\] satisfies
$d_\infty(\sigma,(f_i)_*(\hat{\sigma}))\leq \varepsilon_i \leq 3\varepsilon_i$ and $d_\infty(\hat{\sigma},\sigma^i_\vn) \leq R_i$, so we conclude that $\overline{B}_{R_i}(\sigma_\vn)\subset \overline{B}_{3\varepsilon_i}(\overline{B}_{R_i}(\sigma_\vn^i))$. 

Thus, $(f_i)_*$ is a $3\varepsilon_i$-approximation from $\overline{B}_{R_i}(\sigma_\vn^i)$ to $\overline{B}_{R_i}(\sigma_\vn)$.
\end{proof}


\subsection*
{Proof of Theorem A}
The result follows from Proposition \ref{prop:continuity} and Example~\ref{R:NO_CONTINUITY_P_FINITE}.\qed


\begin{rem}
Note that we have only shown that $\DD_\infty$ is sequentially continuous. To show continuity, we must first introduce topologies on $\MetPair$, $\Met_*$, and $\PMet_*$ compatible with the definitions of Gromov--Hausdorff convergence on each of these categories. Herron has done this for $\Met_*$ in  \cite{herron}. The arguments in \cite{herron} may be generalized to $\MetPair$ and $\PMet_*$, allowing to show the continuity of $\DD_\infty$. This has been carried out in \cite{ahumada_che}.
\end{rem}


\section{Geodesicity}\label{s:geodesicity}

In this section, we show that the functor $\DD_p$, with $p\in [1,\infty)$, preserves the property of being a geodesic space and, in the case $p=2$ and assuming $X$ is a proper geodesic space, we characterize geodesics in the space $\DD_2(X,A)$. This section adapts the work of Chowdhury \cite{chowdhury} to the context of general metric pairs.

The following two lemmas are generalizations of \cite[Lemmas 17 and 18]{chowdhury} and the proofs are similar. For a general metric pair $(X,A)$ where $X$ is assumed to be proper, points in $X$ always have a closest point in $A$. Here, however, as opposed to \cite{chowdhury}, such a point is not necessarily unique.
 
\begin{lem}\label{lem:optimal-bijection-projection}
Let $(X,A)\in \MetPair$. Let $\sigma,\tau \in \DD_p(X,A)$ be diagrams, $\phi_k \colon \sigma \to \tau$ be a sequence of bijections such that $\sum_{x\in \sigma} d(x,\phi_k(x))^p\to
d_p(\sigma,\tau)^p$ as $k\to \infty$. Then the following assertions hold:
\begin{enumerate}
\item If $x\in \sigma$, $y\in \tau\setminus A$ are such that $\lim_{k\to \infty} \phi_k(x)=y$, then there exists $k_0\in \NN$ such that $\phi_k(x)=y$ for all $k\geq k_0$.
\item If $x\in \sigma\setminus A$, $y \in A$ are such that $\lim_{k\to\infty} \phi_k (x) = y$, then $d(x,y) = d(x,A)$.
\end{enumerate}
\end{lem}
\begin{proof}
\begin{enumerate}
    \item Since $p\in[1,\infty)$ and $\tau\in \DD_p(X,A)$, there is some $\eps>0$ such that $B_\eps(y)\cap \tau=\{y\}$. Since $\phi_k(x)\in B_\eps(y)\cap \tau$ for sufficiently large $k$, the conclusion follows.
    \item For the sake of contradiction, if $d(x,y)>d(x,A)$, then $d(x,\phi_k(x))> d(x,A)+2\eps$ and $d(\phi_k(x),A)<\eps$ for sufficiently large $k$, where $\eps = (d(x,y)-d(x,A))/3$. 
    This contradicts the fact that $\sum_{x\in \sigma}d(x,\phi_k(x))^p\to d_p(\sigma,\tau)^p$ as $k\to \infty$. \qedhere
\end{enumerate}
\end{proof}

\begin{lem}\label{lem:convergence-bijections}
Let $(X,A)\in \MetPair$ and assume $X$ is a proper metric space. Let $\sigma,\tau \in \DD_p(X,A)$, and let ${\phi}_k\colon {\sigma} \to {\tau}$ be a sequence
of bijections such that $\sum_{x\in \sigma} d(x,\phi_k(x))^p\to
d_p(\sigma,\tau)^p$ as $k\to \infty$. Then there exists a subsequence $\{\phi_{k_l}\}_{l\in\NN}$ and a limiting bijection $\phi_*$ such that $\phi_{k_l} \to\phi_*$ pointwise as $l\to \infty$ and $\sum_{x\in \sigma} d(x,\phi_*(x))^p=d_p(\sigma,\tau)^p$.
\end{lem}
\begin{proof}

Since $d_p(\sigma,\tau)<\infty$, for each point $x\in \sigma\setminus A$ the sequence $\{\phi_k(x)\}_{k\in\NN}$ consists of a bounded set of points in $X$ and at most countably many points in $A$. In particular, thanks to Lemma \ref{lem:optimal-bijection-projection} and the fact that $X$ is proper, and using a diagonal argument, we can assume that for each $x\in \sigma\setminus A$, the sequence $\{\phi_k(x)\}_{k\in\NN}$ is eventually constant equal to some point $y\in \tau\setminus A$ or it is convergent to some point $y\in A$ such that $d(x,y)=d(x,A)$. In any case, we can define $\phi_*\colon \sigma\setminus A\to \tau$ as
\[
\phi_*(x)=\lim_{k\to\infty} \phi_k(x).
\] 
By mapping enough points in $A$ to all the points in $\tau$ that were not matched with points in $\sigma\setminus A$, we get the required bijection $\phi_*\colon \sigma\to \tau$.
\end{proof}

\begin{cor}[Existence of optimal bijections]\label{cor:optimal-bijections}
Let $(X,A)\in \MetPair$ and assume $X$ is a proper space, then for any $\sigma,\sigma'\in \DD_p(X,A)$ there exists an optimal bijection $\phi\colon\sigma\to\tau$, i.e.\ $d_p(\sigma,\tau)^p = \sum_{x\in\sigma} d(x,\phi(x))^p$.
\end{cor}


\begin{defn}
A \emph{convex combination} in $\DD_p(X,A)$ is a path $\xi\colon [0,1]\to \DD_p(x,A)$ such that there exist an optimal bijection $\phi\colon \xi(0)\to \xi(1)$ and a family of geodesics $\{\xi_x\}_{x\in\xi(0)}$ in $X$ such that $\xi_x$ joins $x$ with $\phi(x)$ for each $x\in \xi(0)$ and  $\xi(t)=\mset{\xi_x(t):x\in \xi(0)}$ for each $t\in [0,1]$. Sometimes we also write $\xi = (\phi,\{\xi_x\}_{x\in\xi(0)})$ to indicate $\xi$ is the convex combination with associated optimal bijection $\phi$ and family of geodesics $\{\xi_x\}_{x\in\xi(0)}$.
\end{defn}

With this definition in hand, the proof of geodesicity follows along the lines of \cite[Corollary 19]{chowdhury}.
\begin{prop}\label{prop:geodesicity}
Let $(X,A)\in \MetPair$. 
	If $X$ is a proper geodesic space, then $\DD_{p}(X,A)$ is a geodesic space.
\end{prop}

\begin{proof}
Let $\sigma,\sigma'\in D_p(X,A)$ be diagrams, $\phi\colon \sigma\to\tau$ be an optimal bijection as in Corollary \ref{cor:optimal-bijections} and let $\xi=(\phi,\{\xi_x\}_{x\in\sigma})$ be some convex combination. Then $\xi$ is a geodesic joining $\sigma$ and $\tau$. Indeed, if we consider the bijection $\phi_s^t\colon \xi(s) \to \xi(t)$ given by $\phi_s^t(\xi_x^\phi(s))=\xi_x^\phi(t)$, then
\[
d_p(\xi(s),\xi(t))^p \leq \sum_{x'\in \xi(s)} d(x',\phi_s^t(x')^p = \sum_{x\in \sigma} d(\xi_x^\phi(s),\xi_x^\phi(t))^p=|s-t|^p\sum_{x\in \sigma}d(x,\phi(x))^p = |s-t|^pd_p(\sigma,\tau)^p.
\]
Therefore $\xi$ is a geodesic from $\sigma$ to $\tau$.
\end{proof}


\section{Non-negative curvature}
\label{s:curvature_bounds}
In this section, we prove that the functor $\DD_2$ preserves non-negative curvature in the sense of Definition~\ref{def:alex} (cf.\ \cite[Theorem 2.5]{turner14} and \cite[Theorems 10 and 11]{chowdhury}. On the other hand, it is known that the functor $\DD_p$ does not preserve the non-negative curvature for $p\neq 2$ (see \cite{turner20}). Also, $\DD_p$ does not preserve upper curvature bounds in the sense of CAT spaces for any $p$ (cf. \cite[Proposition 2.4]{turner14} and \cite[Proposition 2.4]{turner20}). Whether the functor $\DD_2$ preserves strictly negative lower curvature bounds remains an open question.
Additionally, observe that we cannot use the usual $\infty$-norm in $\R^{2}$ to get lower curvature bounds on any space of persistence diagrams, as the following result shows.

\begin{prop}
The space $\DD_p(\mathbb R^2,\Delta)$ is not an Alexandrov space for any $p\in[1,\infty]$ when $\mathbb R^2$ is endowed with the metric $d_\infty$.
\end{prop}

\begin{proof}For $p=\infty$, the space $\DD_\infty(\R^2,\Delta)$ is only a pseudometric space, so it cannot be an Alexandrov space. Suppose now that $1\leq p<\infty$. Consider the points $x_1 = (0,5)$, $x_2 = (0,7)$ and $x_3 = (2,6)$, and let $\sigma_i = \mset{x_i}$ for $i = 1,2,3$. We may check that $d_\infty(x_i,x_j) = 2 \leq d_\infty(x_i,\Delta)$ for all $i \neq j$, implying that for each $i \neq j$ there will be a geodesic $\xi_{i,j}\colon [0,2] \to \DD_p(\mathbb{R}^2,\Delta)$ between $\sigma_i$ and $\sigma_j$ such that $\xi_{i,j}(t)$ has only one point for all $t$. Such geodesics are precisely paths of the form $\xi_{i,j}(t) = \mset{\eta_{i,j}(t)}$, where $\eta_{i,j}\colon [0,2] \to (\mathbb{R}^2,d_\infty)$ is a geodesic between $x_i$ and $x_j$. But for each $i \neq j$ we can pick $\eta_{i,j}$ so that $\eta_{i,j}(1) = y = (1,6)$. This implies that, for instance, $\xi_{1,3}(t) = \xi_{2,3}(t)$ for $t \geq 1$ but not for $t < 1$, implying that there is a branching of geodesics at the point $\mset{y}$, which cannot happen in an Alexandrov space. 
\end{proof}

We will use the following lemma,  which does not require any curvature assumptions, to prove this section's main result.


\begin{lem}\label{lem:midpoints}
Let $\xi\colon [0,1]\to \DD_2(X,A)$ be a geodesic. Let $\phi_i\colon \xi(1/2)\to \xi(i)$, $i=0,1$, be optimal bijections. Then $\phi=\phi_1\circ\phi_0^{-1}\colon \xi(0)\to \xi(1)$ is an optimal bijection and, for all $x\in \xi(1/2)$, $x$ is a midpoint between $\phi_0(x)$ and $\phi_1(x)$.
\end{lem}

\begin{proof}
By the triangle inequality, it is clear that
\[
d(\phi_0(x),\phi_1(x))^2\leq 2d(\phi_0(x),x)^2+2d(x,\phi_1(x))^2
\] 
holds for all $x\in\xi(1/2)$. Therefore,
\begin{align*}
    d_2(\xi(0),\xi(1))^2 &\leq \sum_{z\in \xi(0)}d(z,\phi(z))^2\\
    &= \sum_{x\in\xi(1/2)} d(\phi_0(x),\phi_1(x))^2\\
    &\leq \sum_{x\in\xi(1/2)} 2d(\phi_0(x),x)^2+2d(x,\phi_1(x))^2\\
    &=2\sum_{x\in\xi(1/2)}d(\phi_0(x),x)^2+2\sum_{x\in\xi(1/2)}d(x,\phi_1(x))^2\\
    &=2d_2(\xi(0),\xi(1/2))^2+2d_2(\xi(1/2),\xi(1))^2\\
    &=d_2(\xi(0),\xi(1))^2.
\end{align*}
Thus,
\[
d_2(\xi(0),\xi(1))^2 = \sum_{z\in\xi(0)} d(z,\phi(z))^2 = \sum_{x\in\xi(1/2)} d(\phi_0(x),\phi_1(x))^2 
\]
and
\[
d(\phi_0(x),\phi_1(x))^2 = 2d(\phi_0(x),x)^2+2d(x,\phi_1(x))^2
\]
for all $x\in\xi(1/2)$. In particular, $\phi$ is an optimal bijection between $\xi(0)$ and $\xi(1)$, and $x$ is a midpoint between $\phi_0(x)$ and $\phi_1(x)$ for all $x\in \xi(1/2)$.
\end{proof}

\begin{prop}
\label{prop:curvature}
Let $(X,A)\in \mathsf{Met_{Pair}}$. If $X$ is a proper Alexandrov space with non-negative curvature, then, $\DD_{2}(X,A)$ is also an Alexandrov space with non-negative curvature. 
\end{prop}

\begin{proof}
Since $X$ is an Alexandrov space, it is complete and geodesic. Thus, by Theorem~\ref{T:COMPLETENESS},  the space $\DD_2(X,A)$ is complete, and, since $X$ is assumed to be proper,  Proposition~\ref{prop:geodesicity} implies that $\DD_2(X,A)$ is geodesic. Now we must show that $\DD_2(X,A)$ has non-negative curvature.

Let $\sigma_1,\sigma_2,\sigma_3\in\DD_2(X,A)$ be diagrams and $\xi \colon [0,1] \to \DD_2(X,A)$ be a geodesic from $\sigma_2$ to $\sigma_3$. We want to show that the inequality
\[
d_2(\sigma_1,\xi(1/2))^2\geq \frac{1}{2}d_2(\sigma_1,\sigma_2)^2+\frac{1}{2}d_2(\sigma_1,\sigma_3)^2-\frac{1}{4}d_2(\sigma_2,\sigma_3)^2
\] holds. This inequality characterizes non-negative curvature (see, for example,  \cite[Section 2.1]{Ohta}). 

Let $\phi_i\colon \xi(1/2)\to \sigma_i$, $i=1,2,3$, be optimal bijections, and define $\phi = \phi_3\circ \phi_2^{-1}\colon \sigma_2\to \sigma_3$. 
From the formula for the distance in $\DD_2(X,A)$ we observe that the following inequalities hold:
\begin{align*}
d_2(\sigma_1,\xi(1/2))^2 &= \sum_{x\in \xi(1/2)} d(x,\phi_1(x))^2; \\[5pt]
d_2(\sigma_1,\sigma_2)^2 & \leq \sum_{x\in \xi(1/2)} d(\phi_1(x),\phi_2(x))^2;\\[5pt]
d_2(\sigma_1,\sigma_3)^2 & \leq \sum_{x\in \xi(1/2)} d(\phi_1(x),\phi_3(x))^2.
\end{align*} 
Now, since $\curv(X)\geq 0$, we have that
\[
d(x,\phi_1(x))^2\geq \frac{1}{2}d(\phi_1(x),\phi_2(x))^2+\frac{1}{2}d(\phi_1(x),\phi_3(x))^2-\frac{1}{4}d(\phi_2(x),\phi_3(x))^2
\]
for all $x\in \xi(1/2)$. Therefore, thanks to Lemma \ref{lem:midpoints}, 
\begin{align*}
d_2(\sigma_1, \xi(1/2))^2 &= \sum_{x\in \xi(1/2)} d(x,\phi_1(x))^2\\
&\geq \sum_{x\in \xi(1/2)} \frac{1}{2}d(\phi_1(x),\phi_2(x))^2+\frac{1}{2}d(\phi_1(x),\phi_3(x))^2-\frac{1}{4}d(\phi_2(x),\phi_3(x))^2\\
&\geq \frac{1}{2}d_2(\sigma_1,\sigma_2)^2+\frac{1}{2}d_2(\sigma_1,\sigma_3)^2-\frac{1}{4}d_2(\sigma_2,\sigma_3)^2.\qedhere
\end{align*}
\end{proof}

Lemma~\ref{lem:midpoints} implies the following corollary, which one can use to give an alternative proof of Proposition~\ref{prop:curvature} along the lines of the proof for the Euclidean case in \cite{turner14}.

\begin{cor}\label{c:convex combination geodesics}
Let $(X,A)\in \MetPair$ and assume $X$ is a proper geodesic space. Then every geodesic in $\DD_2(X,A)$ is a convex combination.
\end{cor}

\begin{proof}
This argument closely follows the proofs of Theorems 10 and 11 in \cite{chowdhury}. We repeat some of the constructions for the convenience of the reader.

Let $\xi\colon[0,1]\to \DD_2(X,A)$ be a geodesic. We first claim there exists a sequence of convex combinations $\xi_n = (\phi_{n},\{\xi_{x,n}\}_{x\in\xi(0)})$ such that $\xi(i2^{-n})=\xi_n(i2^{-n})$ for each $n\in \mathbb{N}$ and $i\in\{0,\dots,2^n\}$

Indeed, given $n\in\mathbb{N}$, we define $\phi_{n}$ and $\{\xi_{x,n}\}_{x\in\xi(0)}$ as follows. For each $i\in\{1,\dots,2^{n-1}\}$ consider optimal bijections $\phi_{n,i}^{\pm}\colon \xi((2i-1)2^{-n})\to\xi((2i-1\pm 1)2^{-n})$. By Lemma \ref{lem:midpoints}, \[
\phi_{n}=\phi^+_{n,2^{n-1}}\circ (\phi^-_{n,2^{n-1}})^{-1}\circ\dots\circ\phi^+_{n,1}\circ (\phi^-_{n,1})^{-1}\colon \xi(0)\to \xi(1)
\]
is an optimal bijection. Moreover, Lemma \ref{lem:midpoints} implies that, for each $x\in \xi((2i-1)2^{-n})$, there is some geodesic joining $\phi^-_{n,i}(x)$ with $\phi^+_{n,i}(x)$ which has $x$ as its midpoint. This way, starting from some point $x\in \xi(0)$ and following the bijections $\phi^\pm_{n,i}$, we construct a geodesic $\xi_{x,n}$ joining $x$ with $\phi_{n}(x)$.

Now, thanks to Lemma \ref{lem:convergence-bijections}, there is a subsequence of $\{\phi_n\}_{n\in\mathbb{N}}$ which pointwise converges to some optimal bijection $\phi:\xi(0)\to\xi(1)$. Moreover, 
we can extract a further subsequence $\{\phi_{n_k}\}_{k\in \mathbb{N}}$ such that, for fixed dyadic rationals $l2^{-j}$ and $l'2^{-j}$, the sequence of bijections $\xi(l2^{-j})\to \xi(l'2^{-j})$ induced by $\{\phi_{n_k,i}\}_{k\in \mathbb{N}}$ pointwise converge as well. By Arzel\`a--Ascoli theorem and a applying one more diagonal argument, we may assume that for each $x\in\xi(0)$ the sequence $\{\xi_{x,n_k}\}_{k\in \mathbb{N}}$ uniformly converges to some geodesic $\xi_x$ joining $x$ with $\phi(x)$. By the continuity of $\xi$ and $\hat{\xi}=(\phi,\{\xi_x\}_{x\in\xi(0)})$ it easily follows that $\xi(t) = \hat{\xi}(t)$ for each $t\in [0,1]$.
\end{proof}


\begin{rem} 
We note that $\DD_2(X,A)$ cannot in general be an Alexandrov space with curvature bounded below by $\kappa$ for any $\kappa>0$. To see this, let $(X,A)$ be a metric pair, where $X$ is proper and geodesic. For $i \in \{ 1,2,3 \}$, let $x_i \in X \setminus A$ and let $\xi_i\colon [0,1] \to X$ be a constant speed geodesic with $\xi_i(0) \in A$ and $\xi_i(1) = x_i$ of minimal length, i.e.\ of length $d(x_i,A) = \min_{a \in A} d(x_i,a)$; such $\xi_i$ exists since $X$ is proper and $A$ is closed. Suppose that 
\begin{equation} \label{eq:epsroot2}
d(\xi_i(s),\xi_j(t))^2 \geq d(\xi_i(0),\xi_i(s))^2 + d(\xi_j(0),\xi_j(t))^2 \text{ whenever } i \neq j.
\end{equation}
For $i=1,2,3$, let $\sigma_i = \mset{x_i} \in \DD_2(X,A)$. It follows from \eqref{eq:epsroot2} that $d(x_i,x_j)^2 \geq d(x_i,A)^2 + d(x_j,A)^2$ for $i \neq j$, and therefore $d_2(\sigma_i,\sigma_j) = \sqrt{d(x_i,A)^2 + d(x_j,A)^2}$.

It is then easy to see that the path $\eta_{i,j}\colon [0,1] \to \DD_2(X,A)$, where 
\[
\eta_{i,j}(t) = \mset{\xi_i(1-t),\xi_j(t)},
\]
is a constant speed geodesic in $\DD_2(X,A)$ from $\sigma_i$ to $\sigma_j$. But it is then easy to verify, again using \eqref{eq:epsroot2}, that \[ d_2(\sigma_k,\eta_{i,j}(t)) = \sqrt{d(x_k,A)^2 + d(\xi_i(1-t),A)^2 + d(\xi_j(t),A)^2}, \] where $k \notin \{i,j\}$. In particular, it follows that the geodesic triangle in $\DD_2(X,A)$ formed by geodesics $\eta_{1,2}$, $\eta_{2,3}$ and $\eta_{3,1}$ is isometric to the geodesic triangle in $\mathbb{R}^3$ with vertices $(d(x_1,A),0,0)$, $(0,d(x_2,A),0)$ and $(0,0,d(x_3,A))$. It follows that $\DD_2(X,A)$ cannot be $\kappa$-Alexandrov for any $\kappa > 0$.

The condition \eqref{eq:epsroot2} is not hard to achieve: it can be achieved whenever $X$ is a connected Riemannian manifold of dimension $\geq 2$ and $A \neq X$, for instance. Indeed, in that case, if $|\partial A| \geq 3$ then \eqref{eq:epsroot2} is satisfied for any $x_1,x_2,x_3 \in X \setminus A$ with $d(x_i,a_i) \leq \varepsilon/6$, where $a_1,a_2,a_3 \in \partial A$ are distinct elements and $\varepsilon = \min \{ d(a_i,a_j) : i \neq j \}$. On the other hand, if $|\partial A| \geq 2$ then $|A| \leq 2$ since $X$ is connected of dimension $\geq 2$, and so we may pick $x_1,x_2,x_3 \in X \setminus A$ in such a way that $d(x_1,a) = d(x_2,a) = d(x_3,a) = \varepsilon < d(x_i,b)$ for any $i$ and any $b \in A \setminus \{a\}$, where $a \in A$ is a fixed element. It then follows that $\xi_i(0) = a$ for each $i$. Since $\dim X \geq 2$, we may do this in such a way that the angle between $\xi_i$ and $\xi_j$ at $a$ is $> \pi/2$ when $i \neq j$; but then, as a consequence of the Rauch comparison theorem, \eqref{eq:epsroot2} will be satisfied whenever $\varepsilon > 0$ is chosen small enough.
\end{rem}

\begin{rem}
\label{REM:CONVEX_SUBSETS}
Let $X$ be an Alexandrov space and let $K\subset X$ be a convex subset, i.e. such that any geodesic joining any two points in $K$ remains inside $K$ (cf.\ \cite[p.\ 90]{BBI}). It is a direct consequence of the definition that $K$ is also an Alexandrov space with the same lower curvature bound as $X$. In particular, if $(X,A)\in \MetPair$ with $X$ an Alexandrov space of non-negative curvature, and $K\subset X$ is a convex subset with $A\subset K$, then $\DD_2(K,A)$ is an Alexandrov space of non-negative curvature. 
\end{rem}

\subsection*{Proof of Theorem \ref{t:properties_Dp}}
The result follows from Theorem~\ref{T:COMPLETENESS}, Proposition~\ref{prop:separability}, Proposition~\ref{prop:geodesicity} and Proposition~\ref{prop:curvature}.\qed


\section{Spaces of directions: the local geometry of noise}
\label{S:SPACES_OF_DIRECTIONS}
In this section we prove some metric properties of the space of directions $\Sigma_{\sigma_\varnothing}$ at the empty diagram $\sigma_\varnothing\in \DD_2(X,A)$ for $(X,A)\in \MetPair$ with $X$ an Alexandrov space with non-negative curvature. As mentioned in the introduction, the space of directions at the empty diagram in $\DD_2(\R^2,\Delta)$ could be interpreted as controlling the local geometry of small noise perturbations.

\begin{prop}
\label{prop:diameter.space.of.directions}
The space of directions $\Sigma_{\sigma_\vn}$ has diameter at most $ \pi/2$
\end{prop}

\begin{proof}
    Consider $\sigma,\sigma'\in \DD_{2}(X,A)$. We can always consider a bijection $\phi\colon\sigma\to \sigma'$ such that $\phi(a)=A$ for every $a\in \sigma$ different from $A$ and $\phi^{-1}(a')=A$ for every $a'\in \sigma'$ different from $A$. Thus, by definition of the distance function $d_{2}$, we have
\[
	d_{2}(\sigma,\sigma')^2\leq \sum_{a\in \sigma}d(a,A)^2+ \sum_{a'\in \sigma'} d(a',A)^2= d_{2}(\sigma,\sigma_\vn)^2+d_{2}(\sigma',\sigma_\vn)^2.
\]
Therefore,
\[
\cos\widetilde{\angle} \sigma\sigma_\vn \sigma' = \frac{d_{2}(\sigma,\sigma_\vn)^2+d_{2}(\sigma',\sigma_\vn)^2-d_{2}(\sigma,\sigma')^2}{2d_{2}(\sigma,\sigma_\vn)d_{2}(\sigma',\sigma_\vn)} \geq 0,
\] i.e. $\widetilde{\angle}\sigma\sigma_\vn \sigma' \leq \pi/2$. This immediately implies the result.
\end{proof}

\begin{prop}
\label{prop:directions.to.finite.diagrams}
	Directions in $\Sigma_{\sigma_\vn}$ corresponding to diagrams with finitely many points are dense in $\Sigma_{\sigma_\vn}$.
\end{prop}	
\begin{proof}
Consider an arbitrary diagram $\sigma\in\DD_{2}(X,A)$ and an enumeration $\{a_i\}_{i\in\NN}$ of its points. We can define a sequence of finite diagrams $\{\sigma_n\}_{n\in\NN}$ 
given by 
\[
\sigma_n = \mset{a_1,\dots,a_n}.
\]
Let $\xi$ be a minimizing geodesic joining $\sigma$ with the empty diagram $\sigma_\vn$. By Corollary \ref{c:convex combination geodesics}, we know that $\xi$ is a convex combination, i.e. $\xi= (\phi,\{\xi_x\}_{x\in\sigma})$ for some optimal bijection $\phi\colon\sigma\to\sigma_\varnothing$ and some collection of geodesics $\{\xi_x\}_{x\in\sigma}$ such that $\xi_x$ joins $x\in\sigma$ with $\phi(x)\in A$. Let $\xi_n= (\phi|_{\sigma_n},\{\xi_x\}_{x\in\sigma_n})$ be the restricted convex combination between $\sigma_n$ and $\sigma_\varnothing$. Then the inclusion $i_s\colon \xi_n(s)\to\xi(s)$ induces a bijection between the corresponding diagrams, which in turn implies that
\begin{align*}
d_{2}(\xi_n(s),\xi(s))^2 &\leq \sum_{x\in \xi_n(s)} d(x,i_s(x))^2  \\
&= \sum_{x\in \sigma\setminus\sigma_n} d(\xi_x(s),A)^2 \\
&= s^2\sum_{x\in \sigma\setminus\sigma_n} d(x,A)^2\\
&=s^2(d_2(\sigma,\sigma_\varnothing)^2-d_2(\sigma_n,\sigma_\varnothing)^2).
\end{align*}
Thus, using the definition of the angle between geodesics in an Alexandrov space (see, for example, \cite[Definition 3.6.26]{BBI}) and the law of cosines, we get that
\begin{align*}
1&\geq\cos \angle \sigma_n\sigma_\vn \sigma\\
&= \lim_{s\to 0}\frac{s^2(d_{2}(\sigma_n,\sigma_\vn)^2+d_{2}(\sigma,\sigma_\vn)^2)-d_{2}(\xi_n(s),\xi(s))^2}{2s^2d_{2}(\sigma_n,\sigma_\vn)d_{2}(\sigma,\sigma_\vn)}\\
&\geq \lim_{s\to 0}\frac{s^2(d_{2}(\sigma_n,\sigma_\vn)^2+d_{2}(\sigma,\sigma_\vn)^2-d_2(\sigma,\sigma_\varnothing)^2+d_2(\sigma_n,\sigma_\varnothing)^2)}{2s^2d_{2}(\sigma_n,\sigma_\vn)d_{2}(\sigma,\sigma_\vn)}\\
&=\frac{d_{2}(\sigma_n,\sigma_\vn)}{d_{2}(\sigma,\sigma_\vn)},
\end{align*} 
and the last quotient converges to $1$. Thus, $\angle \sigma_n\sigma_\vn \sigma$ converges to $0$. This way, we can conclude that the set of directions in $\Sigma_{\sigma_\vn}$ induced by finite diagrams can approximate any geodesic direction, and since $\Sigma_{\sigma_\vn}$ is the metric completion of that set, the result follows.
\end{proof}

We can calculate explicitly the angle between any two directions at $\Sigma_{\sigma_\vn}$ determined by finite diagrams, as the following result shows.

\begin{lem} \label{lem:innprodfinite}
Let $\sigma= \mset{a_1,\dots,a_m}$ and $\sigma'=\mset{a'_1,\dots,a'_n}$ be two diagrams with finitely many 
points, and let $\xi_\sigma,\xi_{\sigma'} \colon [0,1] \to \DD_2(X,A)$ be geodesics joining $\sigma_\vn$ to $\sigma,\sigma'$, respectively, so that $\xi_\sigma(t) = \mset{\xi_{a_1}(t),\ldots,\xi_{a_m}(t)}$ and $\xi_{\sigma'}(t) = \mset{\xi_{a_1'}(t),\ldots,\xi_{a_n'}(t)}$ for some geodesics $\xi_{a_i},\xi_{a'_j}\colon [0,1] \to X$ joining $\xi_{a_i}(0) \in A$ to $a_i$ and $\xi_{a_j'}(0)\in A$ to $a_j'$, respectively. Then
\[
d_2(\sigma,\sigma_\vn) d_2(\sigma',\sigma_\vn) \cos \angle(\xi_\sigma,\xi_{\sigma'}) = \max_{\phi\colon\tau \to \tau'} \sum_{a \in \tau} d(a,A)d(\phi(a),A) \cos \angle(\xi_a,\xi_{\phi(a)}),
\]
where $\phi$ ranges over all bijections between subsets $\tau$ and $\tau'$ of 
points in $\sigma$ and $\sigma'$, respectively, such that $\xi_a(0) = \xi_{\phi(a)}(0)$ for all $a \in \tau$.
\end{lem}

\begin{proof}
For each $s,t \in (0,1]$, let $\phi_{s,t}'\colon \xi_\sigma(s) \to \xi_{\sigma'}(t)$ be a bijection realizing the distance $d_2(\xi_\sigma(s),\xi_{\sigma'}(t))$. Then there exists a bijection $\phi_{s,t}$ between subsets $\tau = \tau_{s,t}$ and $\tau' = \tau'_{s,t}$ of 
points in $\sigma$ and $\sigma'$, respectively, such that $\phi_{s,t}'(\xi_x(s)) = \xi_{x'}(t)$ for $x \in \tau$ and $x' = \phi_{s,t}(x) \in \tau'$ and such that $\phi_{s,t}'$ matches all the other points of $\xi_\sigma(s) \cup \xi_{\sigma'}(t)$ to $A$. Moreover, by the construction we have
\begin{align*}
d(\xi_a(0), \xi_{\phi(a)}(0)) - s d(\xi_a(0),a) - t d(\xi_{\phi(a)}(0), \phi(a)) &\leq d(\xi_a(s), \xi_{\phi(a)}(t)) \\
&\leq \left( s^2 d(a,A)^2 + t^2 d(\phi(a),A)^2 \right)^{1/2}
\end{align*}
for all $a \in \tau_{s,t}$ (where $\phi = \phi_{s,t}$), implying that $\xi_a(0) = \xi_{\phi_{s,t}(a)}(0)$ for all $a \in \tau_{s,t}$ when $s$ and $t$ are small enough (which we will assume from now on).

Now we can compute that
\[
d_2(\xi_\sigma(s),\xi_{\sigma'}(t))^2 = s^2 \sum_{a \in \sigma \setminus \tau} d(a,A)^2 + t^2 \sum_{a' \in \sigma' \setminus \tau'} d(a',A)^2 + \sum_{a \in \tau} d(\xi_a(s),\xi_{\phi(a)}(t))^2,
\]
and therefore
\begin{equation} \label{eq:s2t2d}
\begin{aligned}
&s^2 d_2(\sigma,\sigma_\vn)^2 + t^2 d_2(\sigma',\sigma_\vn)^2 - d_2(\xi_\sigma(s),\xi_{\sigma'}(t))^2 \\
&\qquad\qquad{} = \sum_{a \in \tau} \left( s^2 d(a,A)^2 + t^2 d(\phi(a),A)^2 - d(\xi_a(s),\xi_{\phi(a)}(t))^2 \right),
\end{aligned}
\end{equation}
where $\tau = \tau_{s,t}$ and $\phi = \phi_{s,t}$. Moreover, note that since $\phi_{s,t}'$ minimizes $d_2(\xi_{\sigma}(s),\xi_{\sigma'}(t))$, the bijection $\phi\colon \tau\to\tau'$ maximizes the right hand side of \eqref{eq:s2t2d}. It follows that
\begin{equation} \label{eq:ddcos}
\begin{aligned}
&d_2(\sigma,\sigma_\vn) d_2(\sigma',\sigma_\vn) \cos \angle(\xi_\sigma,\xi_{\sigma'}) = \lim_{s,t \to 0} \frac{s^2 d_2(\sigma,\sigma_\vn)^2 + t^2 d_2(\sigma',\sigma_\vn)^2 - d_2(\xi_\sigma(s),\xi_{\sigma'}(t))^2}{2st} \\
&\qquad\qquad{}= \lim_{s,t \to 0} \sum_{a \in \tau_{s,t}} \frac{s^2 d(a,A)^2 + t^2 d(\phi_{s,t}(a),A)^2 - d(\xi_a(s),\xi_{\phi_{s,t}(a)}(t))^2}{2st} \\
&\qquad\qquad{}= \lim_{s,t \to 0} \max_{\phi\colon \tau\to\tau'} \sum_{a \in \tau} \frac{s^2 d(a,A)^2 + t^2 d(\phi(a),A)^2 - d(\xi_a(s),\xi_{\phi(a)}(t))^2}{2st},
\end{aligned}
\end{equation}
where $\phi$ ranges over all bijections between subsets $\tau$ and $\tau'$ of 
points in $\sigma$ and $\sigma'$, respectively, such that $\xi_a(0) = \xi_{\phi(a)}(0)$ for all $a \in \tau$. Since $\sigma$ and $\sigma'$ each has finitely many 
points, there are only finitely many such bijections $\phi$, allowing one to swap the limit and the maximum on the last line of \eqref{eq:ddcos}. The result follows.
\end{proof}

\subsection*{Proof of Theorem~\ref{t:space.of.directions.empty.diagram}} Propositions~\ref{prop:diameter.space.of.directions} and \ref{prop:directions.to.finite.diagrams} correspond to the first two assertions in Theorem~\ref{t:space.of.directions.empty.diagram}.
Using Lemma~\ref{lem:innprodfinite} and the density of the directions in $\Sigma_{\sigma_\vn}$ corresponding to diagrams with finitely many points yields the third assertion in the theorem.
\qed


\section{Dimension of spaces of Euclidean persistence diagrams}\label{s:Euclidean diagrams}

In this section, we analyze some aspects of the global geometry of the spaces of Euclidean persistence diagrams. We denote such spaces by $\DD_{p}(\R^{2n},\Delta_n)$, $1\leq p<\infty$ and $1\leq n\in\mathbb N$, where we let  $\Delta_n = \{ (v,v)\in \R^{2n} : v \in \mathbb{R}^n \}$ (and for simplicity we write $\Delta=\Delta_1$) and  $\mathbb R^{2n}$ is endowed with the Euclidean metric. The investigation of the geometric properties of the spaces $\DD_p( \R^{2},\Delta)$, where the metric in $\R^{2}$ is induced by the $\infty$-norm in $\mathbb R^2$, was carried out in \cite{mileyko1}. In \cite{turner14}, the authors showed that $\DD_2( \R^{2},\Delta)$, where $\R^{2}$ has the Euclidean metric, is an Alexandrov space of non-negative curvature. 

We will also consider the sets
\begin{align*}
\R^{2n}_{\geq 0} &= \{(x_1,\ldots,x_n,y_1,\dots,y_n)\in\R^{2n}: 0\leq x_i\leq y_i,\ i=1,\dots,n\}
\end{align*}
and
\begin{align*}
\R^{2n}_{+} &= \{(x_1,\dots,x_n,y_1,\ldots,y_n)\in\R^{2n}: x_i\leq y_i,\ i=1,\dots,n\},
\end{align*}
which are convex subsets of the Euclidean space $\R^{2n}$. In particular, the space  $\DD_2(\mathbb R^2_{\geq 0},\Delta)$, is the classical space of persistence diagrams which arises in persistent homology, is also an Alexandrov space of non-negative curvature (cf.\ Remark~\ref{REM:CONVEX_SUBSETS}). 
The interest in studying the spaces $\DD_p( \R^{2n},\Delta_n)$ when $n\geq 2$ is motivated by the fact that the subspaces $\DD_p(\mathbb R^{2n}_{\geq0},\Delta_n)\subset \DD_p( \R^{2n},\Delta_n)$ can be thought of as the parameter spaces of a family of $n$-dimensional persistence modules, namely, persistent rectangles (cf.~\cite[Theorem 4.3]{bak2021},  \cite[Lemma 1]{SC17}).

As an application of our geometric results, we now show that the asymptotic dimension of $\DD_{p}(\mathbb R^{2n},\Delta_n)$, $\DD_{p}(\mathbb R^{2n}_+,\Delta_n)$, and $\DD_{p}(\mathbb R^{2n}_{\geq0},\Delta_n)$ is also infinite, for any $1\leq p<\infty$. It may be feasible to also obtain these results  by extending the work of Mitra and Virk in \cite{MitraVirk}, where they consider spaces of persistence diagrams with finitely many points in $\R^2$.  The asymptotic dimension, introduced by Gromov in the context of  finitely generated groups (see \cite{Gromov}),  is a large scale geometric version of the covering dimension. For an introduction to this invariant, we refer the reader to \cite{Bell,BellDran,NoYu,Roe}.

\begin{defn}[\protect{cf. \cite[Definition 2.2.1]{NoYu}}]
Let $\mathcal{U} =  \{U _i\}_{i\in I}$ be a cover of a metric space $X$. Given $R > 0$, the \emph{$R$-multiplicity of $\mathcal{U}$} is the smallest
integer $n$ such that, for every $x\in X$, the ball $B(x,R)$ intersects at most $n$ elements of $\mathcal{U}$. The \textit{asymptotic dimension} of $X$, which we denote by $\asdim X$, is the smallest non-negative integer $n$ such that, for every $R > 0$, there exists a uniformly bounded cover $\mathcal{U}$ with $R$-multiplicity $n+ 1$. If no such integer exists, we let $\asdim X=\infty$.
\end{defn}

The following lemma should be compared with \cite[Lemma 3.2]{MitraVirk}, where the authors compute the asymptotic dimension of spaces of persistence diagrams with $n$ points.

\begin{lem} \label{lem:asdim-ray}
The asymptotic dimension of $\DD_{p}([0,\infty),\{0\})$, $1\leq p<\infty$, is infinite.
\end{lem}

\begin{proof}
Consider the subspace $\DD_p^N([0,\infty),\{0\}) \subset \DD_p([0,\infty),\{0\})$ consisting of diagrams with $\leq N$ 
points. As a set, $\DD_{p}^N([0,\infty),\{0\})$ can be identified with the quotient $[0,\infty)^N / S_N$, where the symmetric group $S_N$ acts by permutations of coordinates. Consider two diagrams $\sigma = \mset{a_1,\ldots,a_{N}}$ and $\sigma' = \mset{a_1',\ldots,a_{N}'}$ in $\DD_p^N([0,\infty),\{0\})$, where $a_1 \geq \cdots \geq a_N \geq 0$ and $a_1' \geq \cdots \geq a_N' \geq 0$. We then claim that
\begin{equation} \label{eq:d2sigma}
d_p(\sigma,\sigma')^p = \sum_{i=1}^N |a_i-a_i'|^p.
\end{equation}
Indeed, by regarding $\sigma$ and $\sigma'$ as atomic measures in $[0,\infty)$ with the same total mass, and applying the classical theory of optimal transport in dimension one, the monotone map $\phi\colon a_i\mapsto a_i'$ induces an optimal bijection between $\sigma$ and $\sigma'$. See for example \cite[Theorem 2.9]{santambrogio}. But this implies that the metric $d_p$ on $\DD_p^N([0,\infty),\{0\})$ agrees with the quotient metric on $([0,\infty)^N,\|\cdot\|_p) / S_N$, where $\|\cdot\|_p$ denotes the $\ell^p$ metric. This implies that the inclusion $\DD_p^N([0,\infty),\{0\})$ into $\DD_p([0,\infty),\{0\})$ is isometric.

Finally, we claim that the asymptotic dimension of $\DD_p^N([0,\infty),\{0\})$ is $N$. Indeed, $([0,\infty)^N,\|\cdot\|_p)$ equipped with the metric is a quotient of an action of $(\mathbb{Z}/2\mathbb{Z})^N$ on $(\mathbb{R}^N,\|\cdot\|_p)$ by isometries, and $\DD_p^N([0,\infty),\{0\})$ is a quotient of an action of $S_N$ on $([0,\infty)^N,\|\cdot\|_p)$ by isometries. As $(\mathbb{R}^N,\|\cdot\|_p)$ and $([0,\infty)^N,\|\cdot\|_p)$ are proper, it follows by \cite[Theorem~1.1]{kas17} that the asymptotic dimensions of $(\mathbb{R}^N,\|\cdot\|_p)$, $([0,\infty)^N,\|\cdot\|_p)$ and $\DD_p^N([0,\infty),\{0\})$ are the same. Thus the asymptotic dimension of $\DD_p^N([0,\infty),\{0\})$ is $N$, as claimed. As $\DD_p^N([0,\infty),\{0\})$ is an isometric subspace of $\DD_p([0,\infty),\{0\})$ for each $N$, it follows that $\DD_p([0,\infty),\{0\})$ has infinite asymptotic dimension, as required.
\end{proof}

\begin{prop} \label{cor:asdim-rayemb}
Let $(X,A) \in \MetPair$ and let $C \geq 1$. Suppose that there exists a $C$-bi-Lipschitz map $f\colon [0,\infty) \to X$ such that $f^{-1}(A) = \{0\}$ and such that $d_X(f(x),A) \geq x/C$ for all $x \in [0,\infty)$. Then $\DD_{p}(X,A)$, $1\leq p <\infty$, has infinite asymptotic dimension.
\end{prop}

\begin{proof}
Note that $f$ induces a map of pairs $f\colon ([0,\infty),\{0\}) \to (X,A)$, and therefore a map 
\[
    f_*\colon \DD_{p}([0,\infty),\{0\}) \to \DD_{p}(X,A).
\]
We will show that $f_*$ is a $C$-bi-Lipschitz equivalence onto its image. The result will then follow from Lemma~\ref{lem:asdim-ray}.
By Proposition~\ref{P:D_p_IS_FUNCTOR}, the map $f_*$ is $C$-Lipschitz. Now, let $\phi'_0\colon f_*(\sigma) \to f_*(\sigma')$ be a bijection realizing the distance $d_{p}(f_*(\sigma),f_*(\sigma'))$, and note that $\phi'_0(f'(x)) = f'(\phi'(x))$ for some bijection $\phi'\colon \sigma \to \sigma'$, where $f'(x) = f(x)$ for $x > 0$ and $f'(0) = A$. Given any $x \in \sigma$, we then have \[|x-\phi'(x)| \leq C \cdot d_X(f(x),f(\phi'(x))),\]
since $f$ is $C$-bi-Lipschitz. Furthermore, if $\phi'(x) = 0$, then we have 
\[
|x-\phi'(x)| = x \leq C \cdot d_X(f(x),A),
\]
and, if $x = 0$, we have 
\[
|x-\phi'(x)| = \phi'(x) \leq C \cdot d_X(f(\phi'(x)),A).
\]
It follows that 
\[
|x-\phi'(x)| \leq C \cdot d_X(f'(x),\phi'_0(f'(x))
\]
in any case, and therefore
\[
d_{p}(\sigma,\sigma')^{p} \leq \sum_{x \in \sigma} |x-\phi'(x)|^{p} \leq C^{p} \sum_{y \in f_*\sigma} d_X(y,\phi'_0(y))^{p} = (C \cdot d_{p}(f_*\sigma,f_*\sigma'))^{p}.
\]
Hence $f_*$ is $C$-bi-Lipschitz, as required.
\end{proof}

Before proving the next result, we recall the definition of \emph{covering dimension}.

\begin{defn}[\protect{cf. \cite[Chapter 8]{mun}}]
Let $\mathcal{U} =  \{U _i\}_{i\in I}$ be an open cover of a metric space $X$. The \emph{order} of  $\mathcal{U}$ is the smallest number $n$ for which each point $p\in X$ belongs to at most $n$ elements in $\mathcal{U}$. The \textit{covering dimension} of $X$ is the minimum number $n$ (if it exists) such that any finite open cover $\mathcal{U}$ of $X$ has a refinement $\mathcal{V}$ of order $n+1$.
\end{defn}

\begin{cor}
\label{C:INFINITE_DIMENSIONS}
The spaces $\DD_{p}(\mathbb R^{2n},\Delta_n)$, $\DD_{p}(\mathbb R^{2n}_+,\Delta_n)$ and $\DD_{p}(\mathbb R^{2n}_{\geq0},\Delta_n)$, for $1\leq p<\infty$, have infinite Hausdorff, covering and asymptotic dimensions. 
\end{cor}

\begin{proof}
For each $X \in \{ \mathbb{R}^{2n}, \mathbb{R}^{2n}_+, \mathbb{R}^{2n}_{\geq 0} \}$, the map $f\colon [0,\infty) \to X$ defined by 
\[
f(x) = \frac{1}{\sqrt{n}}(\overbrace{0,\ldots,0}^n,\overbrace{x,\ldots,x}^n)
\]
is an isometric (and so $\sqrt{2}$-bi-Lipshitz) embedding such that $f^{-1}(\Delta_n) = \{0\}$ and $d_X(f(x),\Delta_n) = x/\sqrt{2}$. Hence,  by Proposition
~\ref{cor:asdim-rayemb}, $\DD_{p}(X,\Delta_n)$ has infinite asymptotic dimension.

To see that $\DD_{p}(X,\Delta_n)$ has infinite covering and Hausdorff dimensions, observe that the same argument in the end the proof of Lemma~\ref{lem:asdim-ray}, shows that the covering and Hausdorff dimensions of $\DD_{p}([0,\infty),{0})$ is infinite. Since $\DD_{p}([0,\infty),{0})\subset \DD_{p}(X,\Delta_n)$, we conclude that $\DD_{p}(X,\Delta_n)$ also has infinite covering and Hausdorff dimensions. 
\end{proof}

Putting the results in this section together yields the proof of our article's last main result. Before proceeding, recall that the \emph{Assouad dimension} of a metric space $X$, when infinite, yields an obstruction to bi-Lipschitz embedding $X$ into a finite-dimensional Euclidean space (see \cite{Fraser} for a detailed discussion of this dimension and  related results). More precisely, if $X$ has a bi-Lipschitz embedding into some finite-dimensional Euclidean space, then $X$ must have finite Assouad dimension (see \cite[Ch.\ 13]{Fraser}). The \emph{Assouad--Nagata dimension}, which Assouad introduced in \cite{Assouad}, may be thought of as a variant of the asymptotic dimension (see \cite{LT} for basic properties of this dimension). 

\subsection*{Proof of Theorem \ref{t:dimensions}}
The result for the Hausdorff, covering, and asymptotic dimensions follows from Corollary~\ref{C:INFINITE_DIMENSIONS}.
Both the Hausdorff and covering dimensions are lower bounds for the Assouad dimension (see \cite{Fraser}), while the asymptotic dimension is a lower bound for the Assouad--Nagata dimension (see \cite{LT}). Therefore, these dimensions are also infinite.\qed

\begin{rem}
Recall that the Hausdorff dimension of an Alexandrov space must be either an integer or infinite (see Section \ref{S:PRELIMINARIES}). Using this fact, we can give an alternative proof that $\DD_2(\R^{2n},\Delta_n)$, $n\geq 1$, has infinite Hausdorff dimension. Indeed, the space $\DD_2(\R^{2n},\Delta_n)$ is not locally compact, since one can always construct sequences of points in arbitrarily small balls around $\Delta_n$ with no convergent subsequence (cf. \cite[Example 16]{mileyko1}). Since an Alexandrov space of finite Hausdorff dimension must be locally compact (see \cite[Theorem 10.8.1]{BBI}), the Hausdorff dimension of $\DD_2(\R^{2n},\Delta_n)$ must be infinite.
\end{rem}

We point out that our arguments to prove Lemma~\ref{lem:asdim-ray}, Proposition~\ref{cor:asdim-rayemb}, and Corollary~\ref{C:INFINITE_DIMENSIONS} can be used to prove analogous results for the spaces of persistence diagrams with finitely (but arbitrarily) many points, $\DD_p^F(X,A)$, as defined, for example, in \cite{bubenik2}. Thus, all such spaces also have infinite Hausdorff, covering, asymptotic Assouad, and Assouad--Nagata dimensions.

\begin{cor}
\label{c:dimensions.finite.spaces}
The space $\DD^F_p(\R^{2n},\Delta_n)$, $1\leq n$ and $1\leq p<\infty$, has infinite covering, Hausdorff, asymptotic, Assouad, and Assouad--Nagata dimension.    
\end{cor}


\appendix

\section{Completeness, separability, and Fr\'echet means}
\label{app:completeness.separability.frechet.means}

\subsection{Completeness and separability}
In this subsection we will show that completeness and separability are both preserved by the functor $\DD_p\colon \mathsf{Met_{Pair}\to \mathsf{Met}_*}$, $1\leq p <\infty$, as well as the existence of Fr\'echet means for probability measures with compact support or with rate of decay at infinity bounded below by $\max\{2,p\}$. The proofs in this section follow almost verbatim the arguments in \cite[Subsection 3.1]{mileyko1} for the properties of the classical spaces of persistence diagrams (i.e.\ when $X=\mathbb R^2_+$ and $A=\Delta$ in our notation). We include these arguments in our more general setting for the sake of completeness. Also, we will assume throughout this section that $p<\infty$, since the results do  not hold in the case $p=\infty$ (see \cite{CGGGMSV2022}). 

We first prove completeness of $\DD_{p}(X,A)$. This proof will follow by putting together a series of lemmas.

\begin{thm}
\label{T:COMPLETENESS}
Let $(X,A)\in \mathsf{Met_{Pair}}$. If $X$ is complete, then $\DD_{p}(X,A)$ is complete.
\end{thm}

Let $\{\sigma_{n}\}$ be a Cauchy sequence in $\DD_{p}(X,A)$. Let the \emph{multiplicity} of $\sigma$, denoted by $\left|\sigma\right|$, be the number of points of $\sigma$ outside $A$ and, for $\al>0$, let $u_{\al}\colon\DD_p(X,A)\to\DD_p(X,A)$ be the function defined by
\[
	u_{\al}(\sigma) = \mset{x\in \sigma : d(x,A)\geq\al}.
\]
We call $u_{\al}(\sigma)$ the \emph{$\al$-upper part} of $\sigma$. We define in a similar way the \emph{$\al$-lower part} of $\sigma$ by letting $l_{\al}\colon\DD_p(X,A)\to\DD_p(X,A)$ be given by
\[l_{\al}(\sigma) = \mset{x\in \sigma : d(x,A)<\al}.\]
Compare with the definition of $u_\al$ and $l_\al$ in \cite[Section 3.1]{mileyko1}. Observe that, in general, we cannot define the persistence of points in $X$ with respect to $A$ as usual (i.e.\ the difference between the coordinates of the point). However, we can take the distance to $A$ as the notion of persistence, which in the case of the classical space of persistence diagrams (either with the norm $\|\cdot\|_\infty$ or $\|\cdot\|_2$ in $\R^2$) is some multiple of the distance to the diagonal $\Delta$. This will affect the computations in this and the following two sections.

Since the $\al$-upper part of any diagram has finite multiplicity for arbitrary $\al$, it is reasonable to consider the convergence of the $\al$-upper parts of the diagrams $\sigma_{n}$.

\begin{lem}\label{lem:delta-alpha}
Let $\al>0$. There exist $M_{\al}\in\mathbb{Z}_{+}$ and $\delta_{\al}, 0<\delta_{\al}<\al$, such that, for all $\delta\in[\delta_{\al},\al)$, there exists an $N_{\delta}>0$ such that $\left|u_{\delta}(\sigma_{n})\right|=M_{\al}$ whenever $n>N_{\delta}$.
\end{lem}

\begin{proof}[Proof of Lemma 4.2]
For $\delta$ with $0<\delta<\al$, let $M_{\sup}^{\delta} = \lim\sup_{n\to\infty}\left|u_{\delta}(\sigma_{n})\right|$ and let $M_{\inf}^{\delta} = \lim\inf_{n\to\infty}\left|u_{\delta}(\sigma_{n})\right|$. Since $\{\sigma_{n}\}$ is a Cauchy sequence, $d_{p}(\sigma_{n},\sigma_\vn)$ is bounded and this directly implies that $M_{\sup}^{\delta}<\infty$.
Also, if $\delta_{1}>\delta_{2}$, then $\left|u_{\delta_{1}}(\sigma_{n})\right|\leq\left|u_{\delta_{2}}(\sigma_{n})\right|$ which means that $M_{\sup}^{\delta_{1}} \leq M_{\sup}^{\delta_{2}}$ and $M_{\inf}^{\delta_{1}}\leq M_{\inf}^{\delta_{2}}$. Therefore, the limits $M_{\sup} = \lim_{\delta\to\al}M_{\sup}^{\delta}$ and $M_{\inf} = \lim_{\delta\to\al}M_{\inf}^{\delta}$ exist and, moreover, there exists a $\delta_{\al}$ such that $M_{\sup} = M_{\sup}^{\delta}$ and $M_{\inf} = M_{\inf}^{\delta}$ whenever $\delta_{\al}\leq\delta<\al$.
Now suppose that $M_{\inf}<M_{\sup}$. Fix $\delta\in(\delta_{\al},\al)$ and let $\eps = \delta - \delta_{\al} >0$. Let $\{\sigma_{n_{k}}\}$ and $\{\sigma_{n_{l}}\}$ be two subsequences of $\{\sigma_{n}\}$ such that $\left|u_{\delta}(\sigma_{n_{k}})\right| = M_{\sup}$ and $\left|u_{\delta_\al}(\sigma_{n_{l}})\right| = M_{\inf}$. Since $\{\sigma_{n}\}$ is a Cauchy sequence, there exists $C>0$ such that $d_{p}(\sigma_{n_{k}},\sigma_{n_{l}})<\eps$ for all $k,l>C$. By assumption, $\left|u_{\delta}(\sigma_{n_{k}})\right| > \left|u_{\delta_\al}(\sigma_{n_{l}})\right|$, which implies that, for any bijection $\phi\colon \sigma_{n_{k}}\to\sigma_{n_{l}}$, there is a point $x\in \sigma_{n_{k}}$ such that $d(x,A) \geq \delta$ and $d(\phi(x),A)< \delta_\al$.  This means that $d(x,\phi(x))>\eps$, leading to $d_{p}(\sigma_{n_{k}},\sigma_{n_{l}})\geq \eps$, which is a contradiction. We then set $M_{\al} = M_{\sup} = M_{\inf}$.
\end{proof}

Given $\al>0$, let $\sigma_{n}^{\al} = u_{\delta_{\al}}(\sigma_{n})$ 
and $\sigma_{n,\al} = l_{\delta_{\al}}(\sigma_{n})$.

 \begin{lem}
For any $\al>0$, the sequence $\{\sigma_{n}^{\al}\}$ is a Cauchy sequence in $\DD_{p}(X,A)$.
\end{lem}

\begin{proof}[Proof of Lemma 4.3]
Let $\delta_{\al}$ be as in Lemma \ref{lem:delta-alpha} and let $\delta\in(\delta_{\al},\al)$. By Lemma~\ref{lem:delta-alpha}, there is an $N>0$ such that,  for all $n>N$, there is no point $x\in \sigma_{n}$ with $d(x,A)\in[\delta_{\alpha},\delta)$. Let $\eps>0$ and let $\eps_{0} = \min\{\eps,(\delta-\delta_{\al})/2\}$.  
If we increase $N$ so that $d_{p}(\sigma_{m},\sigma_{n})<\eps_{0}$ for all $m,n>N$, then there exists a bijection $\phi\colon \sigma_{m}\to \sigma_{n}$ such that 
\[
	\left(\sum\limits_{x\in \sigma_{m}}d(x,\phi(x))^{p}\right)^{\frac{1}{p}} < \eps_{0} \leq \frac{\delta - \delta_{\al}}{2},
\]
which implies that $\phi(\sigma_{m}^{\al}) = \sigma_{n}^{\al}$. Therefore,
\[
d_{p}(\sigma_{m}^{\al},\sigma_{n}^{\al})\leq\left(\sum\limits_{x\in \sigma_{m}^{\al}}d(x,\phi(x))^{p}\right)^{\frac{1}{p}} < \eps_{0}\leq\eps.\qedhere
\]
\end{proof}

The following lemma shows that the sequence $\{\sigma_{n}^{\al}\}$ converges for arbitrary $\al>0$.

\begin{lem}
For any $\al>0$, there is a diagram $\sigma^{\al}\in\DD_{p}(X,A)$ such that $\lim_{n\to\infty}d_{p}(\sigma_{n}^{\al},\sigma^\al) = 0$, hence $\left|\sigma^{\al}\right| = M_{\al}$ and $u_{\al}(\sigma^{\al}) = \sigma^{\al}$. Moreover, if $\al_{1}>\al_{2}$, we have $\sigma^{\al_{1}}\subset \sigma^{\al_{2}}$.
\end{lem}

\begin{proof}[Proof of Lemma 4.4]
Let $\al>0$, $\delta\in(\delta_{\al},\al)$, and let $N>0$ be such that $\left|\sigma_{n}^{\al}\right| = \left|u_{\delta}(\sigma_{n})\right| = M_{\al}$ for all $n>N$. Let $\eps\in(0,\delta/2)$ and choose a subsequence $\{\sigma_{n_{k}}^{\al}\}$ such that $n_{1}>N$ and $d_{p}(\sigma_{n_{k}}^{\al},\sigma_{m}^{\al}) < 2^{-k}\eps$ for $m\geq n_{k}$. Let $\phi_{k}\colon \sigma_{n_{k}}^{\al}\to \sigma_{n_{k+1}}^{\al}$ be a bijection realizing the $p$-Wasserstein distance between $\sigma_{n_{k}}^{\al}$ and $\sigma_{n_{k+1}}^{\al}$, which, by our choice of $\eps$, maps points outside $A$ to points outside $A$. Let $x^{1},\dots,x^{M_{\al}}$ be points outside $A$ in $\sigma_{n_{1}}^{\al}$ and let $\{x_{k}^{1}\},\dots,\{x_{k}^{M_{\al}}\}$ be sequences such that $x_{1}^{i} = x^{i}$, for $i = 1,\dots,M_{\al}$, and $x_{k+1}^{i} = \phi_{k}(x_{k}^{i})$. By the choice of our diagram subsequence, we get that each $\{x_{k}^{i}\}$ is a Cauchy sequence and we denote the corresponding limits by $\hat{x}^{1},\dots,\hat{x}^{M_{\al}}$ (here we are using the fact that $X$ is complete). Let $\sigma^{\al}$ be the diagram whose points outside $A$ are exactly $\hat{x}^{1},\dots,\hat{x}^{M_{\al}}$, where the multiplicity of each $\hat{x}^i$ is the number of sequences whose limit is $\hat{x}^i$.

For $\eps_{0}$, we choose $K>0$ such that, for all $k> K$, we have $d(x_{k}^{i},\hat{x}^{i})<M_{\al}^{-1/p}\eps_{0}/2$ and $d_{p}(\sigma_{n_{k}},\sigma_{m})<\eps_{0}/2$, for $m\geq n_{k}$. It follows that 
\begin{center}
	$d_{p}(\sigma_{m}^{\al},\sigma^{\al}) \leq d_{p}(\sigma_{m}^{\al},\sigma_{n_{k}}^{\al}) + d_{p}(\sigma_{n_{k}}^{\al},\sigma^{\al}) < \eps_{0}/2 + \eps_{0}/2 = \eps_{0}$.
\end{center}
Hence, $\sigma^{\al}$ is the unique limit of $\sigma_{n}^{\al}$ and does not depend on the choice of bijections $\phi_{k}$, subsequences $\{\sigma_{n_{k}}^{\al}\}$, or $\eps$.

Finally, let $\al_{1}>\al_{2}$. Then points $x\in \sigma_{n}^{\al_{2}}$ such that $x\notin \sigma_{n}^{\al_{1}}$ have $d(x,A)<\delta_{\al_{1}}<\al_{1}$. Repeating the above argument with $\al=\al_{2}, N>0$ such that $\left|\sigma_{n}^{\al_{1}}\right| = \left|u_{\delta_{1}}(\sigma_{n})\right| = M_{\al_{1}}$, and $\left|\sigma_{n}^{\al_{2}}\right| = \left|u_{\delta_{2}}(\sigma_{n})\right| = M_{\al_{2}}$, for $n>N$, where $\delta_{1}\in(\delta_{\al_{1}},\al_{1})$,$\delta_{2}\in(\delta_{\al_{2}},\al_{2})$, and $\eps>0$ such that $\eps<\min\{\delta_{2}/2,(\delta_{1}-\delta_{2})/2\}$ leads to the last statement.
\end{proof}

\begin{lem}
Let $\sigma^{\ast} = \bigcup_{\al>0}\sigma^{\al}$. Then $\sigma^{\ast}\in\DD_{p}(X,A)$ and $\lim_{\al\to 0}d_{p}(\sigma^{\al},\sigma^{\ast}) = 0$.
\end{lem}

\begin{proof}[Proof of Lemma 4.5]
Let $\al>0$ and $n\in\NN$ (big enough) such that $d_{p}(\sigma^{\al},\sigma_{n}^{\al})<1$. Then
\begin{center}
	$d_{p}(\sigma^{\al},\sigma_\vn) \leq d_{p}(\sigma^{\al},\sigma_{n}^{\al}) + d_{p}(\sigma_{n}^{\al},\sigma_\vn) \leq 1+C$
\end{center}
    for some constant $C>0$, since $\{\sigma_{n}\}$ is a Cauchy sequence. Since the right-hand side of the preceding equation is independent of $\al$, we get $d_{p}(\sigma^{\ast},\sigma_\vn)\leq 1+C$.

Finally, note that 
\[
	d_{p}(\sigma^{\al},\sigma^*)^{p} \leq d_{p}(l_{\al}(\sigma^{\ast}),\sigma_\vn)^{p} = \sum\limits_{\substack{x\in \sigma^{\ast} \\ d(x,A)<\al}}d(x,A)^{p} \to 0 \  \text{as} \ \al\to 0. \qedhere
\]
\end{proof}

The last step in the proof of the completeness of $\DD_{p}(X,A)$ is the following lemma.

\begin{lem}
For each $\eps>0$, there exists an $\al_{0}>0$ such that, for all $n\in\NN$ and $\al\in(0,\al_{0}]$, we have $d_{p}(\sigma_{n,\al},\sigma_\vn)<\eps$ and, therefore, $d_{p}(\sigma_{n}^{\al},\sigma_{n})<\eps$.
\end{lem}

\begin{proof}[Proof of Lemma 4.6]
Suppose there is an $\eps >0$ such that, for all $\al>0$, there exists  $n_{\al}\in\NN$ with $d_{p}(\sigma_{n,\al},\sigma_\vn)\geq\eps$. Let $\{\al_{i}\}_{i\in\NN}$ be a sequence of positive values monotonically decreasing to $0$. Since $\al_{i}\to 0$, we have $n_{\al_{i}}\to\infty$ and we find a subsequence $\{\sigma_{n_{i}}\}$ such that $d_{p}(\sigma_{n_{i},\al_{i}},\sigma_\vn)\geq\eps$. Let $\delta\in(0,\eps/4)$ and choose $k\in\NN$ such that $d_{p}(\sigma_{n_{k}},\sigma_{n_{i}}) < \delta$, for all $i\geq k$. Now, pick $j\geq k$ such that $d_{p}(\sigma_{n_{k},\al_{i}},\sigma_\vn)<\delta$ for all $i\geq j$. This implies that
\begin{center}
	$d_{p}(\sigma_{n_{i},\al_{i}},\sigma_{n_{k},\al_{j}}) \geq d_{p}(\sigma_{n_{i},\al_{i}},\sigma_\vn) - d_{p}						(\sigma_\vn,\sigma_{n_{k},\al_{j}}) \geq\eps - \delta >3\delta$.
\end{center}
For $i\geq j$ let $\phi_{i}\colon \sigma_{n_{i}}\to \sigma_{n_{k}}$ be a bijection such that $\sum_{x\in \sigma_{n_{i}}}d(x,\phi_{i}(x))^{p} < 2\delta^{p}$. Then also $\sum_{x\in \sigma_{n_{i},\al_{i}}}d(x,\phi_{i}(x))^{p} < 2\delta^{p}$.

Since $\delta_{\al_{j}}>0$, we can pick $l>j$ such that $\delta_{\al_{j}}>2\al_{i}$ for all $i\geq l$. If we now take $x\in \sigma_{n_{i},\al_{i}}$ such that $\phi_{i}(x)\notin \sigma_{n_{k},\al_{j}}$, we see that
\begin{center}
	$d(x,\phi_{i}(x)) \geq \left|d(x,A) - d(\phi_{i}(x),A)\right| \geq \delta_{\al_{j}}-\al_{i} \geq \al_{i}\geq d(x,A)$,
\end{center}
with $i\geq l$. Now let $\hat{\phi}_{i}\colon \sigma_{n_{i},\al_{i}} \to \sigma_{n_{k},\al_{j}}$ be a bijection such that 
\begin{center}
	$\hat{\phi}_{i}(x) = \begin{cases}
		\phi_{i}(x)&\text{if } x\in \sigma_{n_{i},\al_{i}} \text{ and } \phi_{i}(x)\in \sigma_{n_{k},\al_{j}}\\
		A &\text{if } x\in \sigma_{n_{i},\al_{i}} \text{ and } \phi_{i}(x)\notin \sigma_{n_{k},\al_{j}}
		\end{cases}$
\end{center}
and also points $y\in \sigma_{n_{k},\al_{j}} \text{ with } \phi_{i}^{-1}(y)\notin \sigma_{n_{i},\al_{i}}$ are getting mapped to $A$. Then, for $i\geq l$, we have 
\begin{center}
	\begin{align*}
		\sum\limits_{x\in \sigma_{n_{i},\al_{i}}}d(x,\hat{\phi}_{i}(x))^{p} &= \sum\limits_{\substack{x\in 				\sigma_{n_{i},\al_{i}} \\ \phi_{i}(x)\in \sigma_{n_{k},\al_{j}}}}\hspace*{-0.3cm}d(x,\phi_{i}(x))^{p} 			+ \sum\limits_{\substack{x\in \sigma_{n_{i},\al_{i}} \\ \phi_{i}(x)\notin \sigma_{n_{k},										\al_{j}}}}d(x,A)^{p}
		+ \sum\limits_{\substack{y\in \sigma_{n_{k},\al_{j}} \\ \phi_{i}^{-1}(y)\notin \sigma_{n_{i},				\al_{i}}}}d(y,A)^{p}\\[10pt]
		& \leq \sum\limits_{\substack{x\in \sigma_{n_{i},\al_{i}} \\ \phi_{i}(x)\in \sigma_{n_{k},										\al_{j}}}}\hspace*{-0.3cm}d(x,	\phi_{i}(x))^{p} +\hspace*{-0.3cm} \sum\limits_{\substack{x\in 				\sigma_{n_{i},\al_{i}} \\ \phi_{i}(x)\notin \sigma_{n_{k},\al_{j}}}}\hspace*{-0.3cm}d(x,A)^{p} + \delta^{p}
		< 2\delta^{p} + \delta^{p} = 3\delta^{p}.\\
	\end{align*}
\end{center}
Therefore, if $i\geq l$, we have that $d_{p}(\sigma_{n_{i},\al_{i}},\sigma_{n_{k},\al_{j}})^p < 3\delta^{p}$, which is a contradiction.
\end{proof}

The triangle inequality $d_{p}(\sigma^{\ast},\sigma_{n}) \leq d_{p}(\sigma^{\ast},\sigma^{\al}) + d_{p}(\sigma^{\al},\sigma_{n}^{\al}) + d_{p}(\sigma_{n}^{\al},\sigma_{n})$ together with the aforementioned lemmas finally gives us Theorem~\ref{T:COMPLETENESS}.\\

Let us now prove the separability  of $\DD_p(X,A)$.

\begin{prop}
\label{prop:separability}
Let $(X,A)\in \MetPair$. If $X$ is separable, then $\DD_p(X,A)$ is separable.
\end{prop}

\begin{proof}[Proof of Proposition 4.7]
Let $S$ be a countable dense subset of $X$ and let $\hat{S} \subset \DD_p(X,A)$ be the set of persistence diagrams with finite total multiplicity and with points in $S$, that is,
\[
\hat{S} = \{\sigma \in \DD_p(X,A) : |\sigma| < \infty \text{ and } x \in S\ \text{ for all } x \in \sigma\}.
\]
Let $\sigma \in \DD_p(X,A)$. Then, for each $\varepsilon > 0$, we can find $\al > 0$ such that $d_p (l_\al (\sigma),\sigma_\vn) < \eps/2$. Then, we have
$d_p (\sigma,u_\al (\sigma)) \leq d_p (l_\al (\sigma),\sigma_\vn ) < \eps/2$. Since $S^{|u_\al (\sigma)|}$ is dense in $X^{|u_\al (\sigma)|}$, we can find $\sigma_s \in \hat{S}$ such that $d_p(\sigma_s,u_\al (\sigma)) < \eps/2$. Then, $d_p (\sigma,\sigma_s ) \leq d_p (\sigma,u_\al (\sigma)) + d_p (\sigma_s,u_\al(\sigma)) < \eps$,
which implies that $\hat{S}$ is dense.

Note that $\hat{S} = \cup^\infty_{m=0} \hat{S}_m$, where $\hat{S}_m= \{\sigma \in \hat{S} : |\sigma| = m\}$. Each $\hat{S}_m$ can be embedded into $S^{m}$, thus it is countable. Hence, $\hat{S}$ is countable. 
\end{proof}

\begin{rem}
In \cite{bubenik1}, the authors consider completions of spaces of persistence diagrams in the more general context of pairs $(X,A)$ with $X$ an extended pseudometric space (i.e.\ a space in which the distance between points could also be zero or infinite). They define $\overline{\DD}_p(X,A)$ as the set of countable persistence diagrams such that, up to removing a finite subdiagram, have finite $p$-persistence. Let  $X/A = (X \setminus A) \cup {A}$ denote the quotient set obtained by
collapsing $A$ to a point and let $(X,A/\sim_0)$ be the extended metric space obtained by identifying points with zero distance. 
Proposition 6.16 in \cite{bubenik1} (cf.\ \cite[5th Theorem on p.\ 350]{bubenik-hartsock}) asserts that $(\overline{\DD}_p(\overline{X},A/\sim_0),d_p)$ is complete if and only if $(X/A,\overline{d}_1)$ is complete, where
$\overline{d}_1$ is the metric on on $X/A$ given  by
$d_1(x,y) = \min(d(x, y), ||(d(x, A), d(y, A))||)$. In our context, $X$ is a metric space in the usual sense and, therefore, $(\overline{\DD}_p(\overline{X},A/\sim_0),d_p) = (\DD_p(\overline{X},A),d_p)$. Thus, one may obtain an alternative characterization of the completeness of our $\DD_p(X,A)$ via Proposition 6.16 in \cite{bubenik1} . 
\end{rem}

\subsection{Fr\'echet means}
We now consider the existence of Fr\'echet means for probability measures on $\DD_p(X,A)$. Following the arguments in \cite[Section 3.2]{mileyko1}, we will establish a characterization of totally bounded sets in the space of persistence diagrams $\DD_p(X,A)$, $1\leq p< \infty$ (see Proposition \ref{prop:totally bounded sets}), which is the main ingredient in \cite{mileyko1} to prove the existence of Fr\'echet mean sets for probability measures with compact support. Before carrying on, we make the following elementary observation.


\begin{prop}
If $(X,A)\in \MetPair$ and $X\neq A$, then $\DD_p(X,A)$ is not totally bounded. In particular, $\DD_p(X,A)$ is not compact.
\end{prop}

\begin{proof}
Fix $x\in X\setminus A$ and consider the sequence of diagrams $\{\sigma_n\}$ such that $\sigma_n = n\mset{x}$. Then $d_p(\sigma_n,\sigma_\vn)=n^{1/p}d(x,A)$, which is not bounded. 
\end{proof}

The following definition adapts Definitions 17, 18 and 20 from \cite{mileyko1} to our setting.


\begin{defn}
Let $(X,A)\in \MetPair$ and let $S \subset \DD_p(X,A)$.
\begin{enumerate}
    \item  The set $S$ is \textit{birth-death bounded} if the set $\{x\in X : x\in \sigma \text{ for some }\sigma\in S\}$ is bounded.
    \item The set $S$ is \textit{off-diagonally birth-death bounded} if, for all $\eps>0$, the set $u_\eps(S)$ is birth-death bounded.
    \item The set $S$ is  \emph{uniform} if, for all $\eps>0$, there exists $\delta>0$ such that $d_p(l_\delta(\sigma),\sigma_\vn)\leq \eps$ for all $\sigma\in S$.
\end{enumerate}
\end{defn}

These conditions allow us to characterize totally bounded subsets of the space of diagrams, i.e. subsets $S\subset \DD_p(X,A)$ such that, for each  $\varepsilon > 0$, there exists a finite collection of open balls in $\DD_p(X,A)$ of radius $\varepsilon$  whose union contains $S$. The proof of Proposition \ref{prop:totally bounded sets} is a slight modification of that of \cite[Theorem 21]{mileyko1}. Observe again that our definition for the objects $u_\al$ and $l_\al$ differs slightly from that in \cite{mileyko1} due to our different definition for the persistence of points. 

Recall that a metric space $X$ is \textit{proper} if it satisfies the Heine--Borel property, i.e.\ if every closed and bounded subset of $X$ is compact (equivalently, if every closed ball in $X$ is compact). Note that every proper metric space is complete. 

\begin{prop}\label{prop:totally bounded sets}
Let $(X,A)\in \MetPair$ with $X$ a proper metric space. Then, a set $S \subset \DD_p(X,A)$ is totally bounded if and only if it is bounded, off-diagonally birth-death bounded, and uniform.
\end{prop}

\begin{proof}
First, we prove the ``if'' statement. Assume then that $S\subset \DD_p(X,A)$ is totally bounded. Then, in particular, $S$ is bounded. 
Now, let $\eps > 0$, take $0 < \delta < \eps/2$, and let $B_n = B(\sigma_n ,\delta)$, for $n = 1,\dots, N$, be a collection of balls of radius $\delta$ which cover $S$. For each $\sigma_n$ we can find a ball $C_n\subset X$ such that $x\in C_n$ for $x \in  \sigma_n$ with $d(x,A)\geq \eps$, and $d(x,A) < \eps/2$ for all $x \in \sigma_n$ with $x \not\in C_n$. Let $C$ be a ball containing $C_1\cup\dots\cup C_N$. Also, we can find $\al> 0$ such that $d_p (l_\al (\sigma_n),\sigma_\vn) \leq \eps/4$ for $n = 1,\dots,N$.

Let us prove that $S$ is off-diagonally birth-death bounded. We will proceed by contradiction. Suppose that $\sigma \in B_n$ and there is an $x \in \sigma$ such that $d(x,A) \geq  \eps$ and $x\not\in C_\eps$, where $C_\eps\subset X$ is the ball concentric with $C$ and with radius equal to the radius of $C$ plus $\eps$. Then, for any bijection $\phi\colon \sigma \to \sigma_n$, we have $d(x,\phi(x))>\eps-\eps/2$, which implies that $d_p(\sigma,\sigma_n) \geq \eps/2$. This contradicts the assumption that $\sigma \in B_n$ and implies that $u_\eps (S)$ is birth-death bounded since, for all $\sigma\in u_\eps(S)$ and all $x\in \sigma$, we have proved that $x\in C_\eps$.

To prove that $S$ is uniform, we also proceed by contradiction. Suppose
that $\sigma \in B_n$ and $d_p (l_{\al/2} (\sigma),\sigma_\vn) > \eps$. Consider a bijection $\phi\colon \sigma \to \sigma_n$ and let $\sigma_b$ and $\sigma_t$ be maximal subdiagrams of $l_{\al/2}(\sigma)$ such that $d(\phi(x),A) < \al$ for $x \in \sigma_b$ and $d(\phi(x),A) \geq \al$ for $x \in \sigma_t$. If $d_p(\sigma_b,\sigma_\vn)>\eps/2$, then
\[
\left(\sum_{x\in \sigma_b}d(x , \phi(x))^p\right)^{1/p}
\geq d_p(\sigma_b,\phi(\sigma_b)) \geq d_p(\sigma_b,\sigma_\vn) - d_p (\phi(\sigma_b),\sigma_\vn) > \frac{\eps}{2}-\frac{\eps}{4},
\]
where $\phi(\sigma_b)$ denotes the subdiagram of $\sigma_n$ which coincides with the image of $\sigma_b$ under $\phi$.
Since $\sigma_b$ and $\sigma_t$ do not have common points outside $A$ and $l_{\al/2}(\sigma)$ is the union of $\sigma_b$ and $\sigma_t$,
we have 
\[
d_p(l_{\al/2}(\sigma),\sigma_\vn)^p = d_p (\sigma_b,\sigma_\vn)^p + d_p(\sigma_t,\sigma_\vn)^p.
\]
Thus, if $d_p (\sigma_b ,\sigma_\vn) \leq \eps/2$, then \[d_p (\sigma_t,\sigma_\vn) > \left(\eps^p - 2^{-p} \eps^p\right)^{1/p}\geq \eps/2.\] 
Note also that if $x \in \sigma_t$, then $d(x,\phi(x))\geq \al-\al/2 \geq d(x,A)$. Therefore,
\[\left(\sum_{x\in \sigma_t} d(x,\phi(x))^p\right)^{1/p}\geq
\left(\sum_{x\in \sigma_t} d(x,A)^p\right)^{1/p}
= d_p (\sigma_t ,\sigma_\vn ) > \frac{\eps}{2}.
\]
Thus, for any bijection $\phi \colon \sigma \to \sigma_n$ we have
\[
\left(\sum_{x\in \sigma} d(x,\phi(x))^p\right)^{1/p}>\frac{\eps}{4}.
\]
Therefore, $d_p (\sigma,\sigma_n) \geq \eps/4$, which contradicts our assumption that $\sigma \in B_n$. Consequently, 
\[
d_p (l_{\al/2} (\sigma),\sigma_\vn) \leq \eps
\]
for all $\sigma \in S$, which implies that $S$ is uniform.

We now prove the ``only if'' statement. Assume that $S\subset \DD_p(X,A)$ is bounded, off-diagonally birth-death bounded, and uniform. Given $\eps > 0$, let $\delta> 0$ be such that $d_p (l_\delta (\sigma),\sigma_\vn ) < \eps/2$
for all $\sigma \in S$. Take a ball $C\subset X$ such that, for all $\sigma \in S$ and all $x \in u_\delta (\sigma)$, we have $x\in C$. Since $S$ is bounded, we can also find a constant $M \in \NN$ such that $|u_\delta(\sigma)| \leq M$ for all $\sigma \in S$. On the other hand, since $C$ is a bounded subset of a proper complete space, $C$ is also totally bounded and we can find
points $x_1,\dots, x_N \in C$ such that, for any $x \in C$, we have $d(x,x_n) \leq M^{-1/p} \eps/2$ for some $1\leq n\leq N$. Let $\sigma^\ast$ be the diagram consisting of points $x_n$ with $1 \leq n \leq N$, each with multiplicity $M$ and let $\sigma_1,\dots,\sigma_L$ with $L = (M+1)^N$ be all subdiagrams of $\sigma^\ast$. If $\sigma \in S$, we can find $\sigma_n$ and a bijection $\phi\colon u_\delta (\sigma) \to \sigma_n$ such that
\[
\left(\sum_{x\in u_\delta (\sigma)} d(x,\phi(x))^p\right)^{1/p}<\frac{\eps}{2}.
\]
Let $\overline{\phi}\colon \sigma \to \sigma_n$ be the extension of $\phi$ to $\sigma$ obtained by mapping the points in $l_\delta (\sigma)$ to $A$. Then,
\[
\left(\sum_{x\in \sigma} d(x,\overline{\phi}(x))^p\right)^{1/p}
=
\left(\sum_{x\in u_\delta(\sigma)} d(x,\overline{\phi}(x))^p+\sum_{x\in l_\delta(\sigma)} d(x,\overline{\phi}(x))^p\right)^{1/p}
< 2^{\frac{1}{p}-1}\eps \leq \eps.
\]
Therefore, $d_p (\sigma,\sigma_n ) < \eps$ and we conclude that $S$ is totally bounded.
\end{proof}

We now recall the definition of Fr\'echet mean set of probability measures on a metric space and state it in our setting.


\begin{defn}
Given a Borel probability measure $\mu$ on $\DD_p(X,A)$, the quantity
\[
\mathrm{Var}(\mu) = \inf_{\sigma \in \DD_p(X,A)}\left\lbrace F_\mu (\sigma) = \int_{\DD_p(X,A)} d_p (\sigma, \tau)^2\ d\mu (\tau)\right\rbrace
\]
is the \emph{Fr\'echet variance} of $\mu$ and the \emph{Fr\'echet mean set} of $\mu$, denoted by $F(\mu)$ is the set of points in $\DD_p(X,A)$ that realize $\mathrm{Var}(\mu)$, i.e.
\[
F(\mu) = \{\sigma\in \DD_p(X,A) : F_\mu (\sigma) = \mathrm{Var}(\mu) \}.
\]
\end{defn}

We also recall the definitions of various concepts mentioned in Theorem \ref{t:frechet means}.

\begin{defn}
Let $\mu$ be a Borel probability measure on $\DD_p(X,A)$.
\begin{enumerate}
\item We say that $\mu$ has \emph{finite second moment} if
\[
F_\mu(\sigma) < \infty
\] for any $\sigma\in \DD_p(X,A)$.
\item The \emph{support} of $\mu$ is the smallest closed subset $S$ of $\DD_p(X,A)$ such that $\mu(S)=1$.
\item We say that $\mu$ is \emph{tight} if, for any $\eps>0$, there is a compact subset $S\subset \DD_p(X,A)$ such that $\mu(\DD_p(X,A)\setminus S)<\eps$.
\item We say that $\mu$ has \emph{rate of decay at infinity $q$} if for some (and hence for all) $\sigma_0\in \DD_p(X,A)$ there exist $C>0$ and $R>0$ such that, for all $r\geq R$,
\[
\mu(\DD_p(X,A)\setminus B_r(\sigma_0)) \leq Cr^{-q}.
\] 
\end{enumerate}
\end{defn}

The following lemma is essential for the proof of Theorem \ref{t:frechet means}(1) and is an analog of Lemma 23 in \cite{mileyko1}. We include the proof with the necessary modifications for the reader's convenience.


\begin{lem}\label{l:key-lemma-frechet}
Let $\mu$ be a finite Borel measure on $\DD_p(X,A)$ with finite second moment and compact support $S \subset \DD_p(X,A)$, and let $\{\sigma_n\}_{n\in \NN} \subset \DD_p(X,A)$ be a bounded sequence which is not off-diagonally birth-death bounded or uniform. 
Further, let $C_1 > 1$ and $C_2 > 1$ be bounds on $S$ and $\{\sigma_n\}_{n\in \NN}$, respectively, that is, $d_p (\sigma, \sigma_\vn) \leq C_1$ for all $\sigma \in S$ and $d_p (\sigma_n, \sigma_\vn) \leq C_2$ for all $n\in\NN$. Then there exists $\delta > 0$ (depending only on $\{\sigma_n\}_{n\in \NN}$), a subsequence $\{\sigma_{n_k}\}_{k \in \NN}$, and subdiagrams $\overline{\sigma}_{n_k}$ such that
\[
\int_S d_p (\overline{\sigma}_{n_k}, \sigma)^2\ d\mu (\sigma) \leq
\int_S d_p (\sigma_{n_k}, \sigma)^2 d\mu (\sigma) - \eps_0 \mu (S),
\]
where 
\[
\eps_0 = (2^{s/2}-1)(C_1 + C_2)^{2-s}\delta^s,\quad 
s = \max\{2, p\}.
\]
\end{lem}

With these preliminaries in hand, the proof of Theorem~\ref{t:frechet means} now follows as in the Euclidean case (see \cite[Theorems 24 and 28]{mileyko1}). We include the proof of item (1), as it is brief, and indicate the necessary steps to prove item (2), referring to \cite{mileyko1} for further details.

\begin{proof}[Proof of Lemma 5.6]
First, consider the case when $\{\sigma_n\}_{n\in\NN}$ is not off-diagonally birth-death bounded. Fix $x_0\in X$. Then there exists $0 < \eps < 1$ such that, for any $C>0$ and $N > 0$, there is $n > N$ and $x \in \sigma_n$ satisfying $d(x,A) \geq \eps$ and $d(x,x_0)\geq C$. Take $0 < \delta < \eps/2$ and choose $C_0>0$ such that for all $\sigma \in S$ we have $d(x,x_0)\leq C_0$ for $x \in u_\delta (\sigma)$. Set $C_3 = C_0 + C_1 + C_2 + 1$. Let $\{\sigma_{n_k}\}_{k\in \NN}$ be a subsequence of $\{\sigma_n\}_{n\in\NN}$ such that each $\sigma_{n_k}$ contains a point $x$ with $d(x,A) \geq \eps$ and $d(x,x_0) \geq C_3$, and let $\overline{\sigma}_{n_k}$ be the subdiagram of $\sigma_{n_k}$ obtained by removing all such points $x$. Take $\sigma \in S$ and let $\gamma \colon \sigma_{n_k} \to \sigma$ be a bijection such that
\[
\sum_{x\in \sigma_{n_k}} d(x,\gamma(x))^p\leq d_p(\sigma_{n_k},\sigma)^p + \delta^p .
\]
Note that 
\[d_p (\sigma_{n_k},\sigma)\leq d_p (\sigma_{n_k},\sigma_\vn) + d_p (\sigma,\sigma_\vn) \leq  C_1 + C_2.
\]
Hence,
for any $x \in \sigma_{n_k}$ such that $d(x,x_0)\geq C_3$, we have
\[
d(\gamma(x),x_0) \geq  d(x,x_0)-d(\gamma(x),x) \geq C_3-((C_1+C_2)^p+\delta^p)^{1/p} > C_0,
\] 
since we can take $\delta$ to be sufficiently small. Thus, $\gamma(x) \in  l_\delta (\sigma)$ for $x \in \sigma_{n_k}$ with $d(x,x_0) \geq C_3$ and it follows that, for any $x\in \sigma_{n_k}$ such that $d(x,A)\geq \eps$ and $d(x,x_0)\geq C_3$, we have
\[
d(x,\gamma(x))\geq d(x,A)-d(\gamma(x),A) \geq \eps - \delta > \delta.
\] Let $\overline{\gamma}\colon \overline{\sigma}_{n_k} \to \sigma$ be the bijection obtained from $\gamma$ by pairing points $\gamma(x)$ such that $d(x,A) \geq \eps$ and $d(x,x_0) \geq C_3$ to the diagonal. Then
\begin{align}
\label{eq:frechet}
\sum_{x\in \sigma_{n_k}} d(x,\gamma(x))^p &= \sum_{x\in \overline{\sigma}_{n_k}} d(x,\gamma(x))^p + \sum_{x\in \sigma_{n_k}\setminus \overline{\sigma}_{n_k}} d(x,\gamma(x))^p  \nonumber\\
&\geq \sum_{x\in \overline{\sigma}_{n_k}} d(x,\gamma(x))^p + \delta^p \nonumber\\
&\geq \sum_{x\in \overline{\sigma}_{n_k}} d(x,\overline{\gamma}(x))^p+\delta^p
\end{align}
which implies, after applying the inequalities  in the proof of \cite[Lemma 23]{mileyko1}, that
\begin{align}
\label{eq:frechet2}
\left(\sum_{x\in \sigma_{n_k}} d(x,\gamma(x))^p\right)^{2/p} \geq \left(\sum_{x\in \overline{\sigma}_{n_k}} d(x,\overline{\gamma}(x))^p\right)^{2/p}+\eps_0,
\end{align}
where
\[
\eps_0 = (2^{2/s}-1)(C_1+C_2)^{2-s}\delta^2, \quad s = \max\{2,p\}.
\]
Therefore, after taking infimum with respect to $\gamma$ and integrating with respect to $\mu$ on both sides in inequality~\eqref{eq:frechet2}, we obtain
\[
\int_S d_p(\sigma_{n_k},\sigma)^2\ d\mu(\sigma)
\geq \int_S d_p(\overline{\sigma}_{n_k},\sigma)^2\ d\mu(\sigma) + \eps_0 \mu(S).
\]
This proves the lemma in the case where 
$\{\sigma_n\}_{n\in\NN}$ is not off-diagonally birth-death bounded.

Suppose now that $\{\sigma_n\}_{n\in\NN}$ is not uniform. Let $\eps > 0$ be such that, for any $\alpha > 0$ and $N > 0$, there exists $n > N$ such that $d_p (l_\alpha (\sigma_n),\sigma_\vn ) \geq \eps$. If necessary, decrease the $\delta$ from the previous case so that $0 < \delta < \eps/4$ and choose $\alpha_0$ such that $d_p(l_{\alpha_0} (\sigma),\sigma_\vn) \leq \delta$ for all $\sigma \in S$. Take $M \geq 1$ and $C > \delta$ such that, for all $\sigma \in S$, we have $|u_{\alpha_0} (\sigma)| \leq M$ and $d(x,A) \leq C$ for $x \in \sigma$. Define $f \colon [0,1] \to [0,1]$ as $f(x) = 1 - (1 - x)^p$. Note that $f$ is a continuous, monotonically increasing function and $f(0) = 0$, $f(1) = 1$. Set $\delta_0 = f^{- 1} (M^{- 1} C^{-p} \delta^p)$, and $\alpha_1 = \min\{\delta_0 \alpha_0,M^{-1}/p \delta\}$. Let $\{\sigma_{n_k}\}_{k\in\NN}$ be a subsequence of $\{\sigma_n\}_{n\in\NN}$ such that $d_p (l_{\alpha_1} (\sigma_{n_k}),\sigma_\vn) \geq \eps$, $k \geq 1$, and let $\overline{\sigma}_{n_k} = u_{\alpha_1}(\sigma_{n_k})$. Take $\sigma\in S$ and let $\gamma \colon \sigma_{n_k} \to \sigma$ be a bijection such that
\[
\sum_{x\in \sigma_{n_k}} d(x,\gamma(x))^p \leq d_p(\sigma_{n_k},\sigma)^p + \delta^p.
\]
Let $\widetilde\gamma\colon\overline{\sigma}_{n_k}\to \sigma$ be the bijection obtained from $\gamma$ by pairing points in $\gamma(l_{\alpha_1}(\sigma_{n_k}))$ to the diagonal. For convenience, let 
\begin{align*}
    s_0 & =\overline{\sigma}_{n_k},\\
    s_1 & =\{x\in \sigma_{n_k}:d(x,A) < \alpha_1,\ d(\gamma(x),A) < \alpha_0\},\\
    s_2& =\{x\in \sigma_{n_k}: d(x,A) < \alpha_1,\ d(\gamma(x),A)\geq \alpha_0\}.
\end{align*} Note that
\[
\sum_{x\in s_2}d(x,A)^p\leq M\alpha_1^p\leq\delta^p.
\]
Thus,
\[
\sum_{x\in s_1} d(x,A)^p \geq \eps-\delta^p.
\]
Consequently,
\begin{align*}
d_p(s_1,\sigma_\vn)-d_p(\gamma(s_1),\sigma_\vn) &= \left(\sum_{x\in s_1} d(x,A)^p\right)^{1/p}-\left(\sum_{x\in s_1} d(\gamma(x),A)^p\right)^{1/p}\\
&\geq \eps\left(1-\left(\frac{\delta}{\eps}\right)^p\right)^{1/p}-\delta\\
&\geq \beta\delta
\end{align*}
for any $2\leq\beta\leq (4^p-1)^{1/p}-1$. Thus,
\[
\left(\sum_{x\in s_1} d(x,\gamma(x))^p\right)^{1/p}\geq d_p(s_1,\gamma(s_1))\geq d_p(s_1,\sigma_\vn)-d_p(\gamma(s_1),\sigma_\vn) \geq \beta\delta.
\]
Also,

\begin{align*}
\sum_{x\in s_2} d(x,\gamma(x))^p &\geq \sum_{x\in s_2} (d(\gamma(x),A)-\alpha_1)^p\\
    &= \sum_{x\in s_2} \left(d(\gamma(x),A)^p-d(\gamma(x),A)^pf\left(\frac{\alpha_1}{d(\gamma(x),A)}\right)\right)\\
    &\geq \sum_{x\in s_2} \left(d(\gamma(x),A)^p-C^pf\left(\frac{\alpha_1}{\alpha_0}\right)\right)\\
    &\geq -\delta^p+\sum_{x\in s_2} d(\gamma(x),A)^p.
\end{align*}
We then have
\begin{align*}
    \sum_{x\in \sigma_{n_k}} d(x,\gamma(x))^p &=\sum_{x\in s_0} d(x,\gamma(x))^p+\sum_{x\in s_1} d(x,\gamma(x))^p+\sum_{x\in s_2} d(x,\gamma(x))^p\\
    &\geq \sum_{x\in s_0} d(x,\overline{\gamma}(x))^p + \beta^p\delta^p - \delta^p+\sum_{x\in s_2} d(\gamma(x),A)^p\\
    &\geq \beta^p\delta^p - \delta^p +\sum_{x\in s_0} d(x,\overline{\gamma}(x))^p +\sum_{x\in s_1} d(\gamma(x),A)^p +\sum_{x\in s_2} d(\gamma(x),A)^p\\
    &\geq \delta^p+ \sum_{x\in \overline{\sigma}_{n_k}} d(x,\overline{\gamma}(x))^p.
\end{align*}
Thus, we have arrived at inequality~\eqref{eq:frechet}, and we may finish the argument as in the  previous case. This finishes the proof of the lemma.
\end{proof}

\subsection*{Proof of Theorem \ref{t:frechet means}}
For the proof of item (1), let $S\subset \DD_p(X,A)$ be the support of $\mu$ and let $\{\sigma_{n}\}_{n\in \NN}$ be a sequence $\DD_p(X,A)$ such that $F_{\mu}(\sigma_n)\to \mathrm{Var}(\mu)$. 

We will proceed by contradiction. Suppose that $\{\sigma_n\}_{n\in \NN}$ is not bounded and let 
\[
w_n=\inf_{\sigma\in S}d_p(\sigma_n,\sigma).
\]
Then $\{w_n\}_{n\in \NN}$ is not bounded either. In particular,
\[
F_\mu(\sigma_n) = \int_S d_p(\sigma_n,\sigma)^2\ d\mu(\sigma) \geq w_n^2\mu(S)\longrightarrow \infty,
\] which is absurd. Thus, $\{\sigma_n\}_{n\in \NN}$ is bounded.

Assume now that $\{\sigma_n\}_{n\in\NN}$ is not off-diagonally birth-death bounded or it is not uniform. Then,  by Lemma \ref{l:key-lemma-frechet}, there exist a subsequence $\{\sigma_{n_k}\}_{k\in \NN}$ and subdiagrams $\overline{\sigma}_{n_k}\subset \sigma_{n_k}$ such that
\[
\int_S d_p (\overline{\sigma}_{n_k}, \sigma)^2\ d\mu (\sigma) \leq
\int_S d_p (\sigma_{n_k}, \sigma)^2 d\mu (\sigma) - \eps_0 \mu (S).
\]
Taking the infimum over $k$, we get that
\[
\mathrm{Var}(\mu)\leq \mathrm{Var}(\mu)-\eps_0\mu(S),
\] which is a contradiction. This finishes the proof of item (1).

To prove item (2), one first proves an analog of  \cite[Lemma 27]{mileyko1}, with minor modifications necessary to adapt the Euclidean proof to the general setting of $\DD_p(X,A)$. This lemma then implies the result, in a similar fashion as in the proof of \cite[Theorem 28]{mileyko1}.
\hfill $\square$

\section*{Declarations}

\noindent \textbf{Ethical Approval:} Not applicable.\\
 
\noindent \textbf{Competing interests:} The authors do not have any competing interests. \\

\noindent \textbf{Authors' contributions:} All authors contributed equally to the manuscript.\\
 
\noindent \textbf{Funding}: 
M.C. was funded by CONACYT Doctoral Scholarship No. 769708. F.G.G. was funded in part by research grants MTM2017–85934–C3–2–P from the Ministerio de Economía y Competitividad de España (MINECO) and PID2021-124195NB-C32 from the Ministerio de Ciencia e Innovación (MICINN). L.G. was funded in part by research grants MTM2017–85934–C3–2–P from the Ministerio de Economía y Competitividad de España (MINECO), PID2021-124195NB-C32 from the Ministerio de Ciencia e Innovación (MICINN), QUAMAP - Quasiconformal Methods in Analysis and Applications (ERC grant 834728),
and by ICMAT Severo Ochoa project CEX2019-000904-S (MINECO). I.A.M.S. was funded by the Leverhulme Trust (grant RPG-2019-055).\\
 
\noindent \textbf{Availability of data and materials:} Not applicable.

\bibliographystyle{amsplain}
\bibliography{ALEXPERS_CURRENT}

\end{document}